\documentclass[a4paper,11pt]{article}

\usepackage{url}
\usepackage{color}
\usepackage{multirow}
\usepackage{hhline}
\usepackage{amsmath}
\usepackage{amsthm}
\usepackage{amsfonts}
\usepackage{amssymb}
\usepackage{times}
\usepackage{mathptm}
\usepackage{array}
\usepackage{enumerate}
\usepackage{graphics}
\usepackage{epsfig}
\usepackage{mathrsfs}

\newcommand{\subg}[2]{\langle #1\rangle_{#2}}

\newtheorem{thm}{Theorem}[section]
\newtheorem{lem}[thm]{Lemma}
\newtheorem{prop}[thm]{Proposition}
\newtheorem{cor}[thm]{Corollary}

\theoremstyle{definition}

\newtheorem*{df}{Definition}

\newcommand{\Spec}{{\rm Spec}}
\newcommand{\Ev}{{\rm Ev}}
\newcommand{\tr}{{\rm tr}}
\newcommand{\diam}{{\rm diam}}

\newcommand{\bR}{\ensuremath{\mathbb{R}}}

\newcommand{\bZ}{\ensuremath{\mathbb{Z}}}

\newcommand{\cQ}{\ensuremath{\mathcal{Q}}}

\title{The non-bipartite integral graphs with spectral radius three}

\author{
\begin{tabular}{c}
{\sc Taeyoung CHUNG} \\
[1ex]
{\small 
{\it E-mail address}: {\tt tae7837@postech.ac.kr}} \\
[1ex]
{\sc Jack KOOLEN} \\
[1ex]
{\small 
Department of Mathematics, 
POSTECH, 
Pohang 790-784, Korea} \\
{\small 
{\it E-mail address}: {\tt koolen@postech.ac.kr}} \\
[1ex]
{\sc Yoshio SANO} \\
[1ex]
{\small 
Pohang Mathematics Institute, 
POSTECH, 
Pohang 790-784, Korea} \\
{\small 
{\it E-mail address}: {\tt ysano@postech.ac.kr}} \\
[1ex]
{\sc Tetsuji TANIGUCHI} \\
[1ex]
{\small 
Matsue College of Technology, 
Shimane 690-8518, Japan} \\
{\small 
{\it E-mail address}: {\tt tetsuzit@matsue-ct.ac.jp}} \\
\end{tabular}
} 

\date{}

\begin{document}

\maketitle


\begin{center}
{\bf Dedicated to Professor Dragos Cvetkovi\'c 
on the occasion of his 70th birthday}
\end{center}

\begin{abstract}
In this paper, we classify the connected non-bipartite integral graphs 
with spectral radius three. \\

\end{abstract}

\tableofcontents

\section{Introduction and Main Result}

Let $\Gamma$ be a (simple) graph with $n$ vertices. 
The {\it adjacency matrix} $A(\Gamma)$ of $\Gamma$
is the $n \times n$ matrix indexed by the vertices of $\Gamma$ 
such that $A(\Gamma)_{xy}=1$ when $x$ is adjacent to $y$ 
and $A(\Gamma)_{xy}=0$ otherwise. 
The {\it spectral radius} of $\Gamma$ is 
the largest eigenvalue of the adjacency matrix of $\Gamma$. 
An {\it integral graph} is a graph 
whose adjacency matrix has only integral eigenvalues. 

Integral graphs were introduced by F. Harary and A. J. Schwenk \cite{HS}.
F. C. Bussemaker and D. Cvetkovi\'{c} \cite{BC} and 
A. J. Schwenk \cite{Schw}
classified the cubic connected graphs with integral spectrum 
(up to isomorphism, there are exactly 13 such graphs, 
and 5 of them are non-bipartite), 
building on earlier work by D. Cvetkovi\'{c} \cite{cve}. 
S. Simi\'{c} and Z. Radosavljevi\'{c} \cite{SR}
classified the non-regular non-bipartite integral graphs with maximal 
degree exactly four and there are exactly 13 of them. 
For a survey on integral graphs, see \cite{BCRSS}.

In this paper, we classify the connected non-bipartite integral graphs with 
spectral radius three, 
extending the results of \cite{SR}. 
Our main result is as follows: 

\begin{thm}\label{thm:001}
Let $\Gamma$ be a connected non-bipartite integral graph with 
spectral radius three. 
Then, $\Gamma$ is isomorphic to one of the graphs 
in Figure~\ref{fig:GLG} and Figure~\ref{fig:EG}.
\end{thm}


\begin{figure}[h]
\begin{center}
\begin{tabular}{ccccc}
\includegraphics[scale=0.5]
{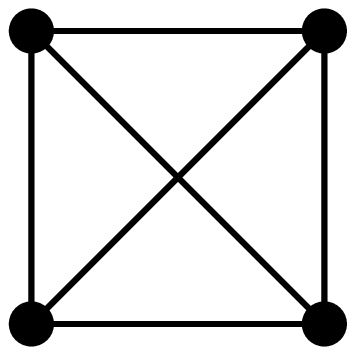} &
\includegraphics[scale=0.5]
{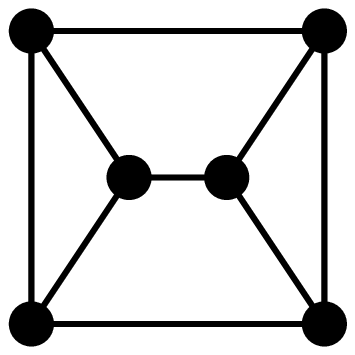} &
\includegraphics[scale=0.5]
{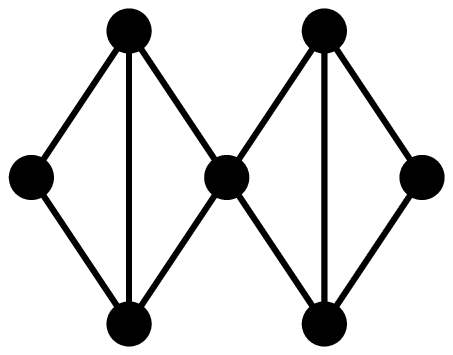} &
\includegraphics[scale=0.5]
{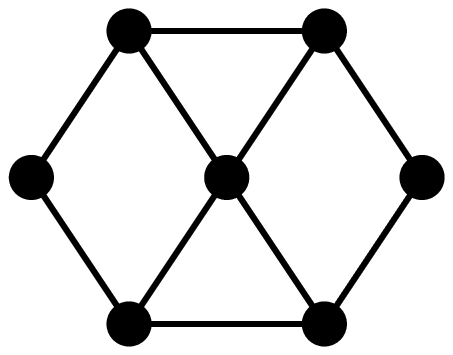} & 
\includegraphics[scale=0.5]
{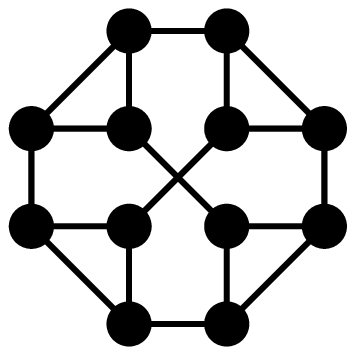} \\ 
LG4 &
LG6 &
LG7a &
LG7b &
LG12 \\
& & & & \\
\includegraphics[scale=0.5]
{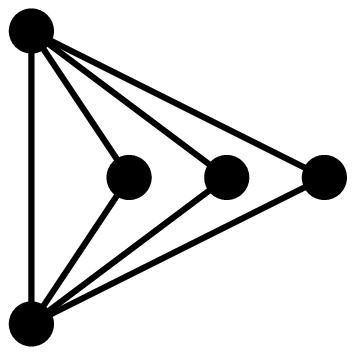} &
\includegraphics[scale=0.5]
{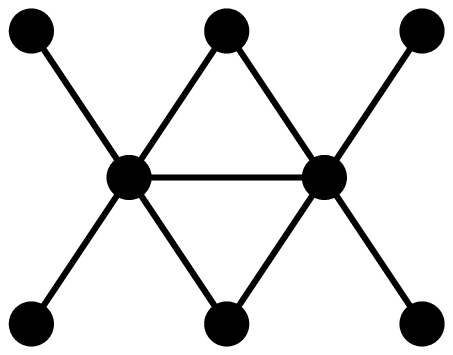} &
\includegraphics[scale=0.5]
{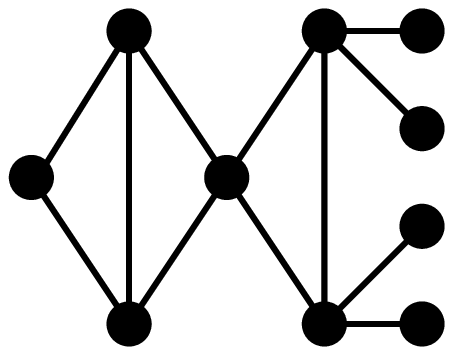} &
\includegraphics[scale=0.5]
{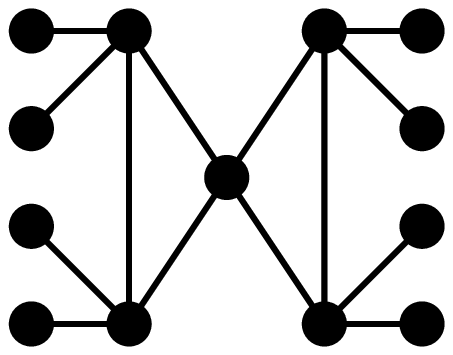} & 
 \\ 
GLG5 &
GLG8 &
GLG10 &
GLG13 & 
 \\
\end{tabular}
\end{center}
\caption{Integral generalized line graphs with spectral radius three}
\label{fig:GLG}
\end{figure}

\begin{figure}[h]
\begin{center}
\begin{tabular}{ccccc}
\includegraphics[scale=0.5]
{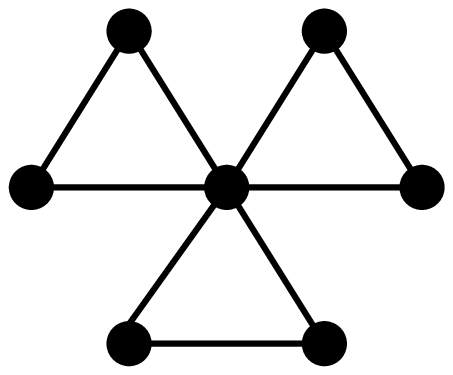} &
\includegraphics[scale=0.5]
{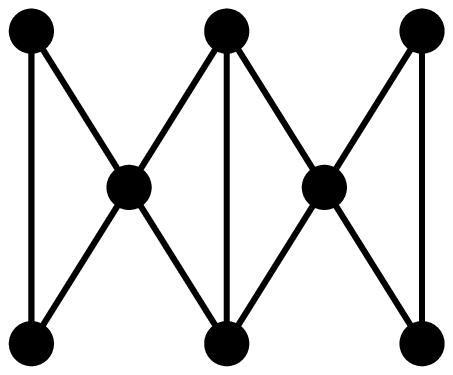} &
\includegraphics[scale=0.5]
{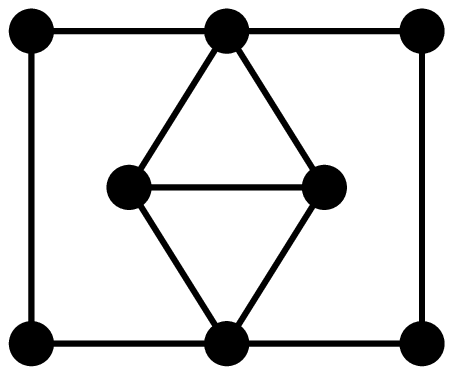} &
\includegraphics[scale=0.5]
{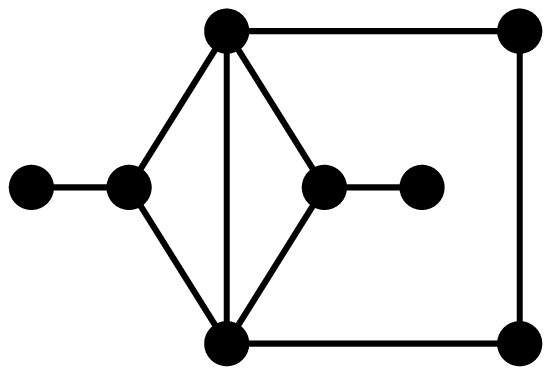} &
\includegraphics[scale=0.5]
{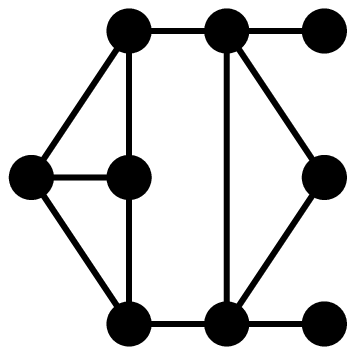} \\
EG7 &
EG8a &
EG8b &
EG8c &
EG9 \\
& & & & \\
\includegraphics[scale=0.5]
{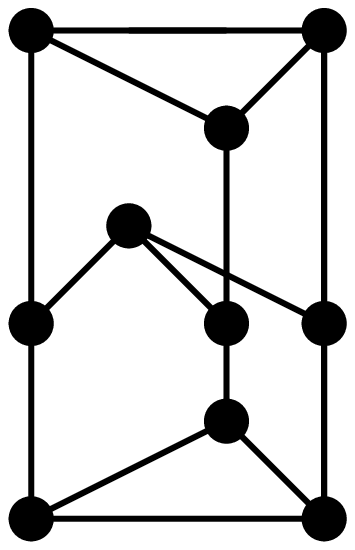} &
\includegraphics[scale=0.5]
{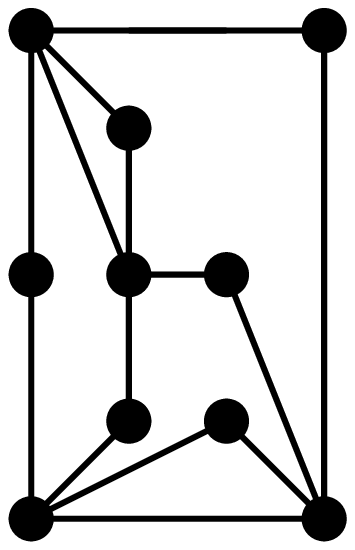} &
\includegraphics[scale=0.5]
{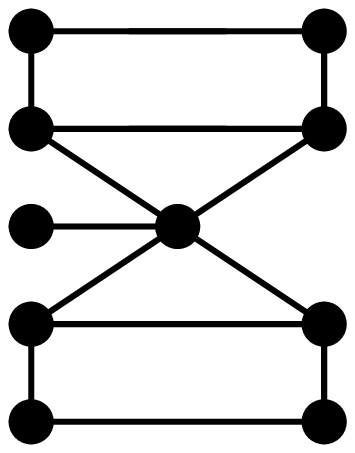} & 
\includegraphics[scale=0.5]
{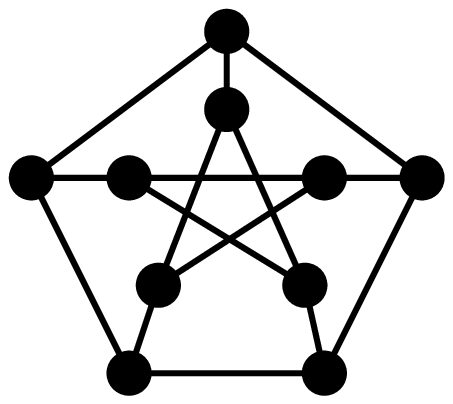} &
\includegraphics[scale=0.5]
{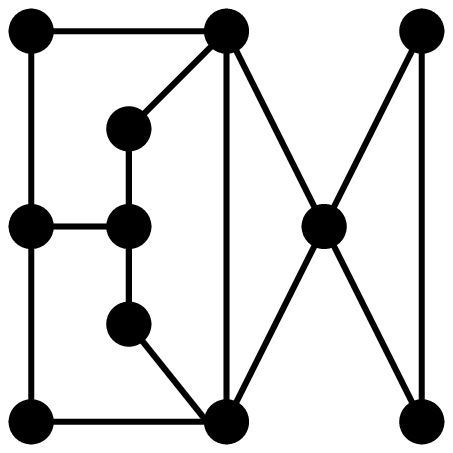} \\
EG10a &
EG10b &
EG10c &
EG10d &
EG11a \\
 & & & & \\
\includegraphics[scale=0.5]
{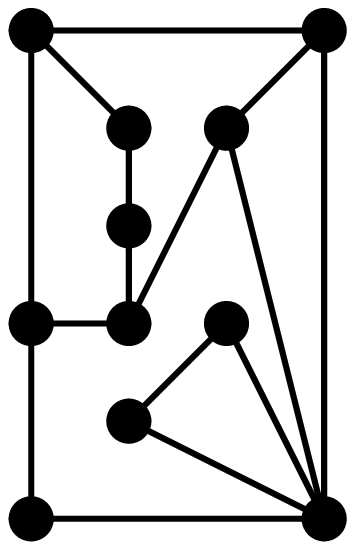} &
\includegraphics[scale=0.5]
{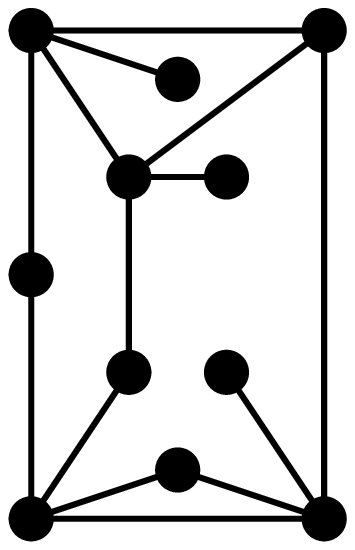} &
\includegraphics[scale=0.5]
{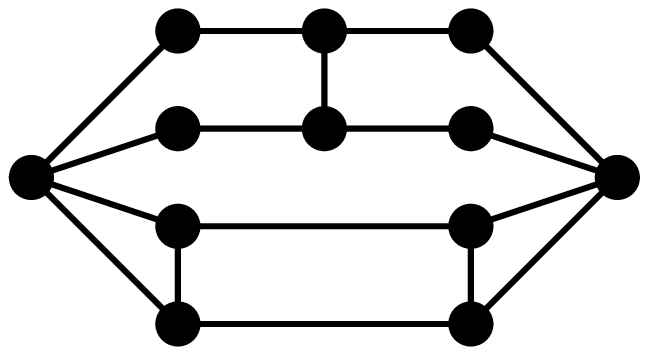} &
 & 
\\
EG11b &
EG11c &
EG12& 
 & 
\\
\end{tabular}
\end{center}
\caption{Integral exceptional graphs with spectral radius three}
\label{fig:EG}
\end{figure}

\begin{table}[h]
\begin{center}
\begin{tabular}{|l|cccccc || l|cccccc|}
\hline
Graph & 3 & 2 & 1 & 0 & -1 & -2 &
Graph & 3 & 2 & 1 & 0 & -1 & -2 \\
\hline
\hline

LG4   & 1 & 0 & 0 & 0 & 3 & 0 &
EG7   & 1 & 0 & 2 & 0 & 3 & 1 \\

LG6   & 1 & 0 & 1 & 2 & 0 & 2 &
EG8a  & 1 & 1 & 1 & 0 & 4 & 1 \\

LG7a   & 1 & 1 & 0 & 1 & 3 & 1 &
EG8b  & 1 & 0 & 3 & 0 & 2 & 2 \\

LG7b   & 1 & 0 & 2 & 1 & 1 & 2 &
EG8c  & 1 & 0 & 2 & 2 & 1 & 2 \\

LG12   & 1 & 3 & 0 & 2 & 3 & 3 &
EG9   & 1 & 1 & 1 & 2 & 2 & 2 \\

GLG5   & 1 & 0 & 0 & 2 & 1 & 1 &
EG10a & 1 & 1 & 3 & 0 & 2 & 3 \\

GLG8   & 1 & 0 & 1 & 4 & 0 & 2 &
EG10b & 1 & 1 & 2 & 2 & 1 & 3 \\

GLG10   & 1 & 1 & 1 & 3 & 2 & 2 &
EG10c & 1 & 1 & 1 & 4 & 0 & 3 \\

GLG13   & 1 & 1 & 2 & 5 & 1 & 3 &
EG10d & 1 & 0 & 5 & 0 & 0 & 4 \\

   &  &  &  &  &  &  &
EG11a & 1 & 1 & 3 & 1 & 2 & 3 \\

   &  &  &  &  &  &  &
EG11b & 1 & 1 & 3 & 1 & 2 & 3 \\

   &  &  &  &  &  &  &
EG11c & 1 & 1 & 2 & 3 & 1 & 3 \\

   &  &  &  &  &  &  &
EG12  & 1 & 2 & 1 & 4 & 0 & 4 \\

\hline
\end{tabular}
\end{center}
\caption{The multiplicities of eigenvalues of graphs}
\label{t:001}
\end{table}

\section{Preliminaries}

In this section, we prepare some notations and terminologies 
which we use in this paper, 
and recall some results on eigenvalues of graphs. 

\subsection{Eigenvalues of graphs}

Let $\Gamma$ be a connected graph where $V(\Gamma)$ is 
the vertex set of $\Gamma$ 
and $E(\Gamma)$ is the edge set of $\Gamma$.
The {\it degree} 
$\deg_{\Gamma}(x)$ 
of a vertex $x$ in $\Gamma$ 
is the number of vertices adjacent to $x$.
Let $d_{\Gamma}(x,y)$ denote the {\it distance} 
between two vertices $x$ and $y$ in $\Gamma$.
The {\it diameter} $\diam(\Gamma)$ of $\Gamma$ 
is the maximum distance between two distinct vertices. 
The {\it degree matrix} $\Delta(\Gamma)$ of $\Gamma$ is 
the diagonal matrix with $\Delta(\Gamma)_{xx}=\deg_{\Gamma}(x)$ 
for any $x \in V(\Gamma)$. 
The {\it Laplace matrix} $L(\Gamma)$ of $\Gamma$ is the matrix 
$\Delta(\Gamma) - A(\Gamma)$. 
The {\it signless Laplace matrix} $Q(\Gamma)$ of $\Gamma$ 
is the matrix $\Delta(\Gamma) + A(\Gamma)$.
Let 
$\Ev(M)$ denote the set of eigenvalues of a matrix $M$. 
Note that if $M$ is a real symmetric matrix, then 
$\Ev(M) \subseteq \mathbb{R}$. 
The {\it spectrum} $\Spec(M)$ of $M$ is 
the multiset of eigenvalues together with their multiplicities.

Before we introduce the Perron-Frobenius Theorem, we need some definitions. 
A real $n \times n$ matrix $M$ with nonnegative entries 
is called {\it irreducible} 
if, for all $i, j$, there exists a positive integer $k$ 
such that $(M^k)_{ij} > 0$. 
For two real $n \times n$ matrices $M$ and $N$, 
we write $N \leq M$ if $N_{ij} \leq M_{ij}$ for all $1 \leq i,j \leq n$. 
We denote the zero matrix by $O$. 

\begin{thm}[{Perron-Frobenius Theorem, cf. \cite[Theorem 8.8.1]{god}}]
\label{perfro}
Let $M$ be an irreducible nonnegative real matrix and 
let $\rho(M) := \max\{ | \theta| \mid \theta \in \Ev(M)\}$. 
Then $\rho(M)$ is an eigenvalue of $M$ 
with algebraic and geometric multiplicity one. 
Moreover, any eigenvector for $\rho(M)$ has either no nonnegative entries 
or no nonpositive entries. 

Let $N$ be a matrix such that $O \leq N \leq M$ 
(in particular $N$ is a principal minor of $M$), 
and $\sigma \in \Ev(N)$. 
Then $- \rho(M) \leq |\sigma| \leq \rho(M)$. 
If $|\sigma| =\rho(M)$, then $N=M$.
\end{thm}

We call $\rho(M)$ defined in the above theorem 
the {\it spectral radius} of $M$. 
(If $M = A(\Gamma)$, then $\rho(M)$ is also called 
the spectral radius of $\Gamma$.)

Let $m \geq n$ be two positive integers. 
Let $M$ be an $m \times m$ matrix 
and let $N$ be an $n \times n$ submatrix of $M$ 
such that $\Ev(M) \subseteq \bR$, $\Ev(N) \subseteq \bR$, and 
both $M$ and $N$ are diagonalizable. 
We say that 
the eigenvalues of $N$ {\it interlace} the eigenvalues of $M$ 
if 
\[
\theta_i(M) \geq \theta_i(N) \geq \theta_{m-n+i}(M)
\]
holds for $i=1, \ldots, n$, where $M$ has eigenvalues 
$\theta_1(M) \geq \theta_2(M) \geq \cdots \geq \theta_m(M)$ 
and $N$ has eigenvalues 
$\theta_1(N) \geq \theta_2(N) \geq \cdots \geq \theta_n(N)$.
We say the interlacing is {\em tight} 
if there exists $\ell \in \{0,1, \ldots, n\}$ 
such that $\theta_i(N)= \theta_i(M)$ 
for $1 \leq i \leq \ell$, 
and $\theta_i(N)= \theta_{m-n+i}(M)$ for $\ell < i \leq n$.

\begin{thm}[{Interlacing Theorem, cf. \cite[Theorem 9.1.1]{god}}]\label{0}
Let $M$ be a square matrix, 
which is similar to a real symmetric matrix, 
and let $N$ be a principal submatrix of $M$. 
Then the eigenvalues of $N$ interlace the eigenvalues of $M$. 
\end{thm}

Let $M$ be a matrix indexed by the vertex set of a graph $\Gamma$ 
and let $\Gamma'$ be an induced subgraph of $\Gamma$. 
We denote by $M|_{\Gamma'}$ 
the principal submatrix of $M$ obtained by restricting the index set 
$V(\Gamma)$ to $V(\Gamma')$. 
A consequence of Perron-Frobenius Theorem is:

\begin{cor}
Let $M$ be a matrix indexed by the vertex set of a graph $\Gamma$ 
and let $\Gamma'$ be a proper induced subgraph of $\Gamma$. 
Then $\theta_{\max}(M) > \theta_{\max}(M|_{\Gamma'})$, 
and $\theta_{\min}(M) \leq \theta_{\min}(M|_{\Gamma'})$. 
where $\theta_{\max}(M)$ and $\theta_{\min}(M)$ denote 
the largest and smallest eigenvalues of $M$, respectively. 
\end{cor}

Let $\pi = \{C_1, C_2, \cdots, C_t\}$ be a partition of 
the vertex set of a graph $\Gamma$. 
The {\it characteristic matrix} of $\pi$ 
is the $|V(\Gamma)| \times |\pi|$ matrix $P$ 
with the characteristic vectors of the elements of $\pi$ as its columns, 
i.e., $P_{xi}=1$ if $x \in C_i$ and $P_{xi}=0$ otherwise. 
If $P$ is the characteristic matrix of $\pi$, 
then $P^TP$ is a diagonal matrix where $(P^TP)_{ii} = |C_i|$. 
Since the parts of $\pi$ are not empty, the matrix $P^TP$ is invertible. 
Let $M$ be a matrix indexed by the vertex set of $\Gamma$. 
The {\it quotient matrix} $B_{{M},\pi}$ of $M$ with respect to $\pi$ 
is defined by $B_{{M},\pi} := (P^TP)^{-1} P^T {M} P$. 
A partition $\pi$ is called {\it $M$-equitable} 
if, for any $1 \leq i, j \leq t$ and any $x \in C_i$, 
$(MP)_{xj} = (B_{{M},\pi})_{ij}$.

\subsection{Generalized line graphs and generalized signless Laplace matrices}

The {\it line graph} $\mathscr{L}(H)$ of a graph $H$ 
is the graph whose vertex set is the edge set of $H$ 
and where two distinct edges of $H$ are adjacent in $\mathscr{L}(H)$ 
if and only if they are incident in $H$.

Now, we recall the definition of generalized line graphs 
which were introduced by Hoffman \cite{Hoff} 
(cf. \cite[Definition 1.1.6]{new}).
A {\it vertex-weighted graph} $(H, f)$ 
is a pair of a graph $H$ and a function 
$f:V(H) \to \mathbb{Z}_{\geq 0}$. 
For $n \in \mathbb{Z}_{> 0}$, 
the {\it cocktail party graph} $CP(n)$ 
is the complete $n$-partite graph $K_{n \times 2}$ 
each of whose patite sets has the size two. 
We let $CP(0)=(\emptyset, \emptyset)$ for convention. 

\begin{df}[\cite{Hoff}]
{\rm 
Let $(H, f)$ be a vertex-weighted graph 
where $f:V(H) \to \mathbb{Z}_{\geq 0}$. 
The {\it generalized line graph} $\mathscr{L}(H, f)$ 
of $(H, f)$ is the graph obtained from 
$\mathscr{L}(H) \cup \bigcup_{x \in V(H)} CP(f(x))$ 
by adding edges between any vertices in $CP(f(x))$ 
and $e \in V(\mathscr{L}(H))$ such that 
$x \in e$ in $H$. 
A graph $\Gamma$ is called a {\it generalized line graph} 
if there exists a vertex-weighted graph $(H, f)$ 
such that $\Gamma \cong \mathscr{L}(H,f)$. 
}
\end{df}

In 1976, Cameron, Goethals, Seidel, and Shult \cite{root} showed 
the following theorem: 

\begin{thm}\label{thm:Cameron}
Let $\Gamma$ be a connected graph with smallest eigenvalue at least $-2$. 
Then, $\Gamma$ is a generalized line graph or $\Gamma$ is a graph 
with at most $36$ vertices. 
\end{thm}

\noindent
A connected graph with smallest eigenvalue at least $-2$ 
is called {\it exceptional} if it is not a generalized line graph.

For a function $f:V \to \mathbb{Z}_{\geq 0}$, 
we denote the sum 
$\sum_{x \in V} f(x)$ by $|f|$. 
We denote the function $f:V \to \mathbb{Z}_{\geq 0}$ 
such that $f(x)=0$ for all $x \in V$ 
simply by $0$. 

Now, we introduce the generalized signless Laplace matrix 
of a vertex-weighted graph. 

\begin{df}
{\rm 
The {\it generalized signless Laplace matrix} 
$\cQ(H, f)$ of a vertex-weighted graph $(H,f)$ 
is the square matrix of size $|V(H)|$ defined 
by 
\[
\cQ(H, f) := Q(H) + 2\Delta_f = A(H) + \Delta(H) + 2\Delta_f, 
\]
where 
$\Delta_f$ is the diagonal matrix defined by 
$(\Delta_f)_{xx}=f(x)$ for any $x \in V(H)$. 
}
\end{df}

We will see that the generalized signless Laplace matrix $\cQ(H,f)$ 
plays a similar role for the generalized line graph $\mathscr{L}(H,f)$ 
as the signless Laplace matrix $Q(H)$ for the line graph $\mathscr{L}(H)$ 
(see \cite{CRS2}, \cite{towards1}, \cite{towards2}, \cite{towards3} 
for recent research on signless Laplacians). 
Note that $\cQ(H,0)=Q(H)$ by definition. 

\begin{df}
{\rm 
For a vertex-weighted graph $(H,f)$, 
we define the {\it incidence matrix} 
$N_{(H, f)}$ of $(H,f)$ 
by 
\[
N_{(H, f)} := 
\begin{pmatrix}
N_{H} & N_f & N_f \\
O & I_{|f|} & - I_{|f|}
\end{pmatrix}
,
\]
where 
$N_{H}$ is the vertex-edge incidence matrix of $H$ 
and $N_f$ is the $\{0,1\}$-matrix of size $|V(H)| \times |f|$ 
such that each column has exactly one nonzero entry 
and that each row indexed by $x \in V(H)$ has exactly $f(x)$ 
nonzero entries. 
}
\end{df}

\begin{prop}\label{propinc}
Let $(H,f)$ be a vertex-weighted graph and $\Gamma := \mathscr{L}(H, f)$ 
be the generalized line graph of $(H,f)$, 
and $N:=N_{(H, f)}$ be the incidence matrix of $(H,f)$. 
Then 
\begin{eqnarray*}
N^TN &=& A(\Gamma) + 2I_{|E(H)|+2{|f|}}, \\ 
NN^T &=& 
\begin{pmatrix}
Q(H) + 2N_f N_f^T & O \\
O & 2 I_{|f|} 
\end{pmatrix} 
 =  
\begin{pmatrix}
\cQ(H,f) & O \\
O & 2 I_{|f|} 
\end{pmatrix}. \\
\end{eqnarray*}
\end{prop}

In the rest of this section, we collect some results 
on generalized signless Laplace matrices.

\begin{prop}\label{prop:diam}
Let $(H,f)$ be a connected vertex-weighted graph. 
If $H$ has diameter $D$, 
then $\cQ(H, f)$ has at least $D + 1$ distinct eigenvalues. 
\end{prop}

\begin{proof}
Let $\cQ:=\cQ(H, f)$. 
The set $\{I, {\cQ}, {\cQ}^2, \ldots, {\cQ}^D \}$ 
consists of linearly independent matrices. 
Therefore $\cQ$ has at least $D+1$ distinct eigenvalues 
(as $\cQ$ is diagonalizable). 
\end{proof}

We can show the following 
by the same proof as \cite[Lemma 9.6.1]{god}. 

\begin{prop}\label{qint}
Let $\cQ$ be the generalized signless Laplace matrix 
of a vertex-weighted graph $(H,f)$, 
and let $\pi$ be a partition of the vertex set of $H$. 
Then the eigenvalues of the quotient matrix $B_{{\cQ},\pi}$ 
interlace the eigenvalues of $\cQ$. 
Moreover, 
if the interlacing is tight, then $\pi$ is $\cQ$-equitable. 
\end{prop}

\begin{prop}\label{propsemi}
Let $(H,f)$ be a vertex-weighted graph. 
Then, the following hold: 
\begin{itemize}
\item[{\rm (i)}] 
$\cQ(H, f)$ is positive semidefinite. 
\item[{\rm (ii)}]
The multiplicity of $0$ as an eigenvalue of $\cQ(H,f)$ is equal to 
the number of bipartite connected components $C$ of $H$ 
such that the restriction of 
$f$ to $C$ 
is a $0$-function. 
\end{itemize}
\end{prop}

\begin{proof}
(i) Immediately from Proposition \ref{propinc}.\\
(ii) Without loss of generality, we may assume $H$ is connected. 
Assume $ \cQ(H,f)$ has an eigenvalue. 
Let ${\bf x}$ be an eigenvector with the eigenvalue $0$. 
Then $Q(H) {\bf x} = {\bf 0}$ and $\Delta_f {\bf x} = {\bf 0}$. 
By \cite[Proposition 2.1]{CRS2}, $H$ is bipartite. 
Let $H$ have the two color classes $V_R$ and $V_B$, 
and let $K$ be the diagonal matrix with $K_{xx} = 1$ 
if $x \in V_B$ and $-1$ otherwise. 
Then, it is well-known that $Q(H) = K L(H) K$. 
This means that $K {\bf x}$ is an eigenvector for $L(H)$ 
for $\theta_{\max}(L(H))$, 
and therefore we may assume is has only positive entries 
by the Perron-Frobenius theorem. 
But this means that ${\bf x}$ has no zero entry. 
This implies that $\Delta_f$ has to be the $0$-matrix. 
This shows the proposition. 
\end{proof}

\begin{cor}\label{0ev}
Let $(H,f)$ be a connected vertex-weighted graph. 
Then, $0 \in \Ev(\cQ(H,f))$ 
if and only if $H$ is bipartite and $f=0$. 
In this case, 
the multiplicity of $0$ as an eigenvalue of $\cQ(H,f)$ 
is $1$. 
\end{cor}

Proposition \ref{propinc} also implies:

\begin{prop} \label{prop:radius}
Let $(H, f)$ be a vertex-weighted graph and 
$\Gamma:=\mathscr{L}(H, f)$ be the generalized line graph of 
$(H, f)$. 
Then, the following hold: 
\begin{itemize}
\item[{\rm (i)}] 
$\Gamma$ is an integral graph if and only if 
$\cQ(H,f)$ has only integral eigenvalues. 
\item[{\rm (ii)}] 
$\Gamma$ has spectral radius $\rho$ if and only if 
$\cQ(H,f)$ has spectral radius $\rho + 2$. 
\end{itemize}
\end{prop}

\section{Proof of Theorem \ref{thm:001}}

In this section, we prove Theorem \ref{thm:001}. 
If $\Gamma$ has spectral radius three and $\Gamma$ is non-bipartite, 
then $-3 \not\in \Ev(A(\Gamma))$. 
Since $\Gamma$ is an integral graph, 
we have $\Ev(A(\Gamma)) \subseteq \{-2,-1,0,1,2,3\}$. 
By Theorem \ref{thm:Cameron}, 
$\Gamma$ is either a generalized line graph 
or an exceptional graph. 
We deal with the case of generalized line graphs 
in Subsection \ref{sec:proof-GLG} 
and the case of exceptional graphs in Subsection \ref{sec:proof-EG}. 
Then, Theorem \ref{thm:001} follows from 
Theorems \ref{thm:GLG} and \ref{thm:EG001}.

\subsection{The case of generalized line graphs}\label{sec:proof-GLG}

In this subsection we determine 
the connected integral generalized line graphs 
with spectral radius three. We will show:

\begin{thm}\label{thm:GLG}
Let $\Gamma$ be a connected integral generalized line graph 
with spectral radius three. 
Then, $\Gamma$ is one of the $9$ graphs in Figure \ref{fig:GLG}. 
\end{thm}

Let $\Gamma$ be a connected integral generalized line graph 
with spectral radius three, say $\Gamma = \mathscr{L}(H,f)$ 
for some connected vertex-weighted graph $(H, f)$. 
Then the generalized signless Laplace matrix $\cQ(H,f)$ is integral 
and has spectral radius five. 
So, instead of determining the connected integral generalized line graphs 
with spectral radius three, we will first determine 
the connected vertex-weighted graph $(H, f)$ 
whose generalized signless Laplace matrix $\cQ(H, f)$ 
has only integral eigenvalues and spectral radius five. 

First we will give some more general results and 
then we will consider the case where $0 \in \Ev(\cQ(H,f))$ 
and after that we will consider the case $0 \not \in \Ev(\cQ(H,f))$. 
One of the reasons to do so is that $H$ 
usually has much less vertices than $\Gamma$. 
We give a computer-free proof.
It is easy to see that the generalized line graphs of 
the vertex-weighted graphs 
$(H_1,0)$, $(H_2,0)$, $(H_3,0)$, $(H_4,0)$, 
$(H_5, f_5)$, $(H_6, f_6)$, $(H_7, f_7)$, $(H_8, f_8)$, 
$(H_9, f_9)$ in Figures \ref{fig:H1-5} and \ref{fig:H6-9} 
are the graphs 
LG4, LG6, LG7b, LG12, GLG5, GLG8, LG7a, GLG10, GLG13 
in Figure \ref{fig:GLG}, respectively. 
By Proposition \ref{prop:radius}, 
Theorem \ref{thm:GLG} follows from 
Propositions \ref{prop:gr0ev} and \ref{prop:not0ev}.

\subsubsection{General results}

In this subsection we will develop some general results 
to help us in this case of generalized line graphs.  

Let us begin with the following lemma:

\begin{lem}\label{lem:deg-f-pair}
Let $(H,f)$ be a connected vertex-weighted graph with 
$\Ev(\cQ(H,f)) \subseteq \{\theta \in \bR \mid \theta \leq 5 \}$. 
Then, for each $x \in V(H)$, we have 
\[
(\deg_{H}(x),f(x)) \in \{(1,0),(1,1),(2,0),(2,1),(3,0), (4,0)\}. 
\]
Moreover, if there is a vertex $x$ with $(\deg_{H}(x),f(x)) = (4, 0)$, 
then $H = H_1 (=K_{1,4})$ and $f=0$. 
\end{lem}

\begin{proof}
By the Perron-Frobenius Theorem (Theorem \ref{perfro}), 
we have $(\cQ(H, f))_{xx} \leq 5$ 
and $(\cQ(H, f))_{xx} = 5$ only if $|V(H)|=1$. 
Since $(\cQ(H, f))_{xx}= \deg_{H}(x) + 2f(x)$, 
the first part of the lemma holds. 
As $\cQ(K_{1,4}, 0)$ has spectral radius $5$ 
the moreover part follows immediately from the Perron-Frobenius Theorem. 
\end{proof}

For nonnegative integers $i$ and $j$, 
let 
\[
A_{(i,j)}:=\{x \in V(H) \mid (\deg_{H}(x),f(x)) = (i,j) \}
\]
By Lemma \ref{lem:deg-f-pair}, 
if $(H,f) \neq (K_{1,4},0)$, then we have 
\[
V(H) = A_{(1,0)} \cup A_{(1,1)} \cup A_{(2,0)} \cup A_{(2,1)} 
\cup A_{(3,0)}. 
\]
We say a vertex $x$ in $H$ is {\it of type $(i,j)$} if $x \in A_{(i,j)}$. 
Let $a_{(i,j)}$ denote the cardinality of $A_{(i,j)}$.

\begin{lem}\label{substr}
Let $(H,f)$ be a connected vertex-weighted graph with 
$5 \in \Ev(\cQ(H, f)) \subseteq \{0,1,2,3,4,5\}$ 
such that $H$ has maximum degree $3$. 
Then, we have the following: 
\begin{itemize}
\item[{\rm (1)}] 
If $a_{(1,0)} \neq 0$, then $0 \in \Ev(\cQ(H, f))$, 
and hence $H$ is bipartite and $f =0$;
\item[{\rm (2)}] 
If $x, y \in A_{(1,0)}$, then $d_H(x,y) \leq 2$; 
\item[{\rm (3)}] 
$a_{(1,0)} \leq 2$; 
\item[{\rm (4)}] 
If $H$ has two adjacent vertices $x, y \in A_{(2,0)}$ 
and if they do not have a common neighbour, 
then  $0 \in \Ev(\cQ(H, f))$;
\item[{\rm (5)}] 
$H$ does not contain an induced subgraph $H'$ with exactly two 
components each of which is a cycle. 
\end{itemize}
\end{lem}

\begin{proof}
(1) 
Let $x$ be a vertex of degree $1$ and 
let $u$ be its unique neigbour. 
Then the signless Laplace matrix restricted to $\{u, x\}$ 
has smallest eigenvalue less then $1$. This shows (1). \\
(2) 
This shows that for $x,y \in A_{(1,0)}$ we have 
$d_H(x,y) \leq 3$, as $0$ has multiplicity at most one. 
If $x$ and $y$ have distance $3$ 
then let $x, u, v, y$ be a shortest path between $x$ and $y$. 
Now the signless Laplace matrix restricted to $\{x, u, v, y\}$ 
has second smallest eigenvalue less then one, 
which is impossible by interlacing 
as the multiplicity of $0$ is at most one.  \\
(3) 
If $a_1 \geq 3$, then let $x, y, z$ be three vertices of degree $1$. 
Let $u$ be their unique common neighbour. 
But then $H = K_{1,3}$, a contradiction with that 
the multiplicity of $5$ is one. \\
(4) Consider the principal submatrix of $\cQ$ indexed by $x, y, $ 
and the other neighbour of $x$. 
Then this submatrix has smallest eigenvalue smaller then one. 
The statement now immediately follows from interlacing. 
(5) 
As $H$ is connected each cycle has a vertex of degree $3$, 
which implies if we look at the signless Laplace matrix 
with restricted 
to $H'$ then this matrix has two eigenvalues more than $4$, 
a contradiction as by interlacing $m_5 \geq 2$, but $m_5 = 1$. 
This completes the proof.
\end{proof}

Let $(H, f)$ be a connected vertex-weighted graph such that 
$5 \in\Ev(\cQ(H,f)) \subseteq \{0,1,2,3,4,5\}$. 
Let $m_r$ denote the multiplicity of $r \in \bR$ 
as an eigenvalue of $\cQ:=\cQ(H,f)$. 
Since $\Ev(\cQ(H,f)) \subseteq \{0,1,2,3,4,5\}$, 
we have $m_r=0$ for $r \in \bR \setminus \{0,1,2,3,4,5\}$. 
Note that $m_5 =1$ and $m_0 \in \{0, 1\}$. 

By the equations 
$\text{tr}(\cQ^i) = \sum_{r \in \bR} r^i m_r$ for $i = 0,1,2,3$, 
we obtain the following:

\begin{prop}\label{prop:trace}
Let $(H,f)$ be a connected vertex-weighted graph with 
$5 \in \Ev(\cQ(H,f))) \subseteq \{0,1,2,3,4,5\}$. 
Then, the following hold: 
\begin{eqnarray}
\label{eq:m2-0}
m_0 + m_1 + m_2 + m_3 + m_4 + 1 
&=& a_{(1,0)} + a_{(2,0)} + a_{(3,0)} + a_{(1,1)} + a_{(2,1)}, \\
\label{eq:m2-1}
m_1 + 2 m_2 + 3 m_3 + 4 m_4 + 5 
&=& a_{(1,0)} + 2 a_{(2,0)} + 3 a_{(3,0)} + 3 a_{(1,1)} + 4 a_{(2,1)}, \\
\label{eq:m2-2}
m_1 + 4 m_2 + 9 m_3 + 16 m_4 + 25 
&=& 2a_{(1,0)} + 6 a_{(2,0)} + 12 a_{(3,0)} + 10 a_{(1,1)} + 18 a_{(2,1)}. 
\end{eqnarray}
\end{prop}

\begin{proof}
Since $5 \in \Ev(\cQ(H,f))$ and $H$ is connected, we have $m_5=1$. 
By considering the equation 
$\text{tr}(\cQ^i)= \sum_{r \in \bR} r^i m_r$ for $i = 0,1,2,3$, 
we obtain the equations. 
\end{proof}

\begin{cor}\label{cor:tr-8}
Let $(H,f)$ be a connected vertex-weighted graph with 
$5 \in \Ev(\cQ(H,f))) \subseteq \{0,1,2,3,4,5\}$. 
Then, the following holds: 
\begin{equation}\label{eq:2-8}
4 m_0 - 2m_2 - 2m_3 + 4= a_{(1,0)} + a_{(3,0)} - a_{(1,1)} + 2 a_{(2,1)}. 
\end{equation}
\end{cor}

\begin{proof}
By calculating $4 \times$ [Equation (\ref{eq:m2-0})] 
$+ (-5) \times$ [Equation (\ref{eq:m2-1})] 
$+ 1 \times$ [Equation (\ref{eq:m2-2})], 
we obtain Equation (\ref{eq:2-8}). 
\end{proof}

\subsubsection{The case where $0 \in \Ev(\cQ(H,f))$} 

In this subsection, we will show the following result.

\begin{prop}\label{prop:gr0ev}
Let $(H,f)$ be a connected vertex-weighted graph with  
$5 \in \Ev(\cQ(H,f)) \subseteq \{0,1,2,3,4,5\}$. 
If $0 \in \Ev(\cQ(H,f))$, 
then $f=0$ and $H$ is one of the four graphs in Figure \ref{fig:H1-5}.  
\end{prop}

\begin{figure}[h]
\begin{center}
\begin{tabular}{cccc}
\includegraphics[scale=0.5]
{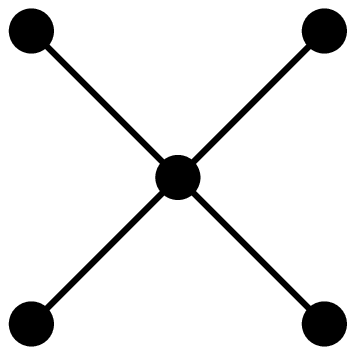} &
\includegraphics[scale=0.5]
{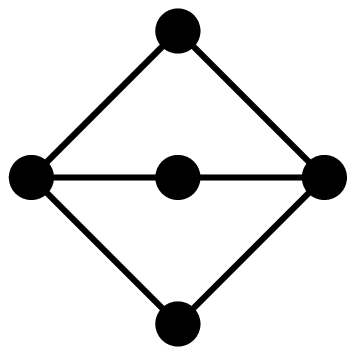} &
\includegraphics[scale=0.5]
{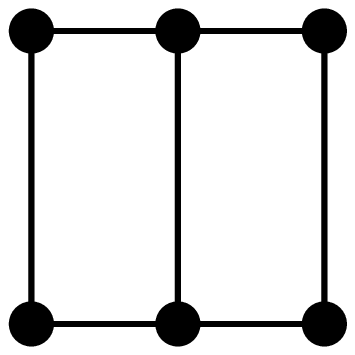} &
\includegraphics[scale=0.5]
{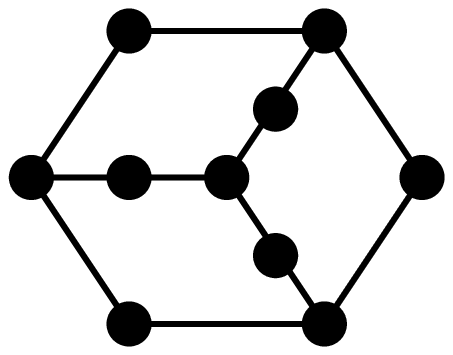} \\
$H_1$ &
$H_2$ &
$H_3$ &
$H_4$ \\
\end{tabular}
\end{center}
\caption{The graphs $H_1, H_2, H_3, H_4$}
\label{fig:H1-5}
\end{figure}

Let $(H,f)$ be a connected vertex-weighted graph
such that 
$\{0, 5\} \subseteq \Ev(\cQ(H,f)) \subseteq \{0,1,2,3,4,5\}$. 
Then, by Corollary \ref{0ev}, 
$H$ is bipartite and $f=0$. 
Although the result in this case now follows from \cite{SR}, 
we will give a computer-free proof.

\begin{lem}\label{lemjack}
Let $H$ be a connected graph with maximum degree at most $3$ 
such that 
$Q(H) \subseteq \{\theta \in \bR \mid \theta \leq 5 \}$. 
Let $H'$ be an induced subgraph of $H$ 
which has no vertex of degree $1$. 
Let $a'_i$ be the number of vertices of $H'$ of degree $i$ for $i=2,3$. 
Let $m$ be the number of edges 
with one endpoint a vertex of degree $2$ 
and the other endpoint a vertex of degree $3$. 
If $a'_3 \neq 0$ and $a'_2 \neq 0$, then 
$1 + \frac{m}{a'_2} \leq \frac{m}{a'_3}$ 
and equality implies that $5 \in \Ev(Q(H))$ 
and $H=H'$.
\end{lem}

\begin{proof}
Let $A'_i$ be the set of vertices of $H'$ of degree $i$ for $i=2,3$.  
Consider the quotient matrix of $Q(H)$ 
\[
B = \begin{pmatrix}
6 - \frac{m}{a_3} & \frac{m}{a_3} \\
\frac{m}{a_2} & 4-\frac{m}{a_2}
\end{pmatrix}
\] 
with respect to the partition $\{ A'_3, A'_2 \}$. 
By interlacing (Proposition \ref{qint}), we obtain that $B$ 
has largest eigenvalue at most $5$. 
This is equivalent to the inequality 
$1 + \frac{m}{a_2} \leq \frac{m}{a_3}$. 
Now equality means that the largest eigenvalue of $Q(H')$ is equal to $5$, 
and hence by Perron-Frobenius Theorem (Theorem \ref{perfro}) 
we obtain $H=H'$. 
This shows the lemma. 
\end{proof}

Let $H'$ be the (induced) subgraph of $H$ 
obtained by consecutively removing degree $1$ vertices from $H$. 
Since $H$ is connected, $H'$ is also connected. 
Since the vertices of $H$ have degree at most $3$, 
the vertices of $H'$ have degree $2$ or $3$. 
Let 
$A'_{(3,0)} := \{x \in V(H') \mid \deg_{H'}(x) = 3 \}$,
$A'_{(2,0)} := \{y \in V(H') \mid \deg_{H'}(y) = 2 \}$,  
$a'_{(3,0)} := |A'_{(3,0)}|$, $a'_{(2,0)} := |A'_{(2,0)}|$, 
and 
$m$ be the number of edges of $H'$ 
with exactly one endpoint in $A'_{(3,0)}$. 
Note that $m$ is also the number of edges of $H'$ 
with exactly one endpoint in $A'_{(2,0)}$.  

\begin{prop}\label{propjack}
Let $H$ be a connected graph with 
$5 \in \Ev(Q(H)) \subseteq \{0,1,2,3,4,5\}$. 
Then the following hold: 
\begin{itemize}
\item[{\rm (1)}] 
If $x,y \in A'_{(3,0)}$, then $d_{H'}(x,y) \leq 3$. 
\item[{\rm (2)}] 
If $a'_{(3,0)} \neq 0$ and $a'_{(2,0)} \neq 0$, then 
$1 + \frac{m}{a'_{(2,0)}} \leq \frac{m}{a'_{(3,0)}}$ 
and equality implies $H = H'$. 
\item[{\rm (3)}] 
If $\frac{m}{a'_{(2,0)}} = 2$, 
then $\frac{m}{a'_{(3,0)}} = 3$ and $H' = H$, 
and moreover $H =H_2(= K_{2,3})$ or $H =H_4$. 
\end{itemize}
\end{prop}

\begin{proof} 
(1) As $Q(K_{1,3})$ has spectral radius $4$, it follows immediately 
from the Perron-Frobenius Theorem (Theorem \ref{perfro}) 
and interlacing (Theorem \ref{0}), 
by considering the subgraph induced on $x$ and $y$ and their neighbours. \\ 
(2) Immediately from Lemma \ref{lemjack}. \\ 
(3) If $\frac{m}{a'_{(2,0)}} = 2$, then by (2), 
we have $\frac{m}{a'_{(3,0)}} = 3$ and $H = H'$. 
As, by interlacing, $3 \geq m_2 \geq a_{(2,0)} - a_{(3,0)} = \frac{n}{5}$, 
we obtain $n$ is one of $5, 10, 15$. But, if $n=15$ then $m_2 =3$ 
and hence $a_{(3,0)} \leq 2,$ a contradiction. 
If $n=10$ then, by (1), any two vertices of degree 3 have a common neighbour. 
This implies $H = H_4$. 
If $n = 5$, it is completely clear.
\end{proof}

Recall that a {\it spanning tree} 
of a connected graph $H$ with $n$ vertices 
is a connected subgraph of $H$ with $(n-1)$ edges and no cycle. 

\begin{prop}[{cf. \cite[Lemma 13.2.4]{god}}] 
Let $H$ be a graph with $n$ vertices and 
$0=\mu_1 \leq \mu_2 \leq \cdots \leq \mu_n$ be 
the eigenvalues of the Laplace matrix $L(H)$ of $H$. 
Then, the number of spanning trees of $H$ is equal to 
$\frac{1}{n}\mu_2 \mu_3 \cdots \mu_n$. 
\end{prop}

As we already have seen in the proof of Proposition \ref{propsemi}, 
if a graph $H$ is bipartite 
then the Laplace matrix $L(H)$ is similar 
to the signless Laplace matrix $Q(H)$ and hence $\Spec(L(H)) = \Spec(Q(H))$.
As a consequence we have:

\begin{cor}\label{cor:235}
Let $H$ be a connected bipartite graph with  
$\Ev(Q(H)) \subseteq \{0,1,2,3,4,5\}$. 
Then, the number of vertices of $H$ has 
a form $2^a \cdot 3^b \cdot 5^c$ 
with $a,b, c \in \bZ_{\geq 0}$ satisfying 
$0 \leq a \leq m_2 + 2m_4$, $0 \leq b \leq m_3$ and $0 \leq c \leq m_5$. 
\end{cor}

\begin{proof}[Proof of Proposition \ref{prop:gr0ev}]
As $f = 0$, we have $a_{(1,1)}=a_{(2,1)}=0$ and $\cQ(H,f)=Q(H)$. 
Note that $n:=|V(H)|=a_{(1,0)} + a_{(2,0)} + a_{(3,0)}$. 
Since $H$ is bipartite, we have $m_0=1$ by Proposition \ref{propsemi}. 
By Corollary \ref{cor:tr-8}, we have 
\begin{equation}\label{eq:8}
-2m_2 -2m_3+8= a_{(1,0)} + a_{(3,0)}. 
\end{equation}
Since $m_2$ and $m_3$ are nonnegative, $a_{(1,0)} + a_{(3,0)} \leq 8$. 
Moreover the nonnegative integer $a_{(1,0)} + a_{(3,0)}$ must be even, 
and so 
$a_{(1,0)} + a_{(3,0)} \in \{0,2,4,6,8\}$. 
If $a_{(3,0)}=0$, then $H$ has maximum degree at most $2$ and hence 
all the row sums of $Q(H)$ are at most $4$, 
so the largest eigenvalue of $Q(H)$ is at most $4$. 
Therefore 
\[
a_{(3,0)} \neq 0 \quad \mbox{ and } \quad 
a_{(1,0)} + a_{(3,0)} \in \{2,4,6,8\}. 
\]
By Equation (\ref{eq:8}), we obtain
\[
0 \leq m_2 + m_3 \leq 3. 
\]

By Proposition \ref{prop:trace}, we have 
\begin{eqnarray*}
m_1 + m_2 + m_3 + m_4 + 2 &=& a_{(1,0)} + a_{(2,0)} + a_{(3,0)}, \\
m_1 + 2 m_2 + 3 m_3 + 4 m_4 + 5 &=& a_{(1,0)} + 2 a_{(2,0)} + 3 a_{(3,0)}, \\
m_1 + 4 m_2 + 9 m_3 + 16 m_4 + 25 
&=& a_{(1,0)} + 6 a_{(2,0)} + 12 a_{(3,0)}. 
\end{eqnarray*}

Since $a_{(1,0)} \leq 2$ by Lemma \ref{substr} (3), 
there are 26 posibilities for $(m_2,m_3,a_{(1,0)},a_{(3,0)})$. 
By solving the system of the above equations 
with given parameters 
$m_2$, $m_3$, $a_{(1,0)}$, and $a_{(3,0)}$, 
we obtain 
$(m_2, m_3, a_{(1,0)}, a_{(3,0)};$ $m_1, m_4, a_{(2,0)})$ 
as in Table \ref{table:para}.

\begin{table}[h]
\begin{center}
\begin{tabular}{|l|cccc | lll || c |}
\hline
Cases  & $m_2$ & $m_3$ & $a_{(1,0)}$ & $a_{(3,0)}$ & 
$m_1$  & $m_4$ & $a_{(2,0)}$ & $|V(H)|$\\
\hline
\hline
(A0)   &  0  &  0  &  0  &  8        & 2t+3 & t+5 & 3t+2      & 3t+10  \\
(B0)   &  1  &  0  &  0  &  6        & 2t+1 & t+3 & 3t+1      & 3t+7    \\
(C0)   &  0  &  1  &  0  &  6        & 2t+2 & t+3 & 3t        & 3t+8   \\
(D0)   &  2  &  0  &  0  &  4        & 2t+1 & t+2 & 3t+3      & 3t+7    \\
(E0)   &  1  &  1  &  0  &  4        & 2t   & t+1 & 3t+1      & 3t+5     \\
(F0)   &  0  &  2  &  0  &  4        & 2t+1 & t+1 & 3t+2      & 3t+6    \\
(G0)   &  3  &  0  &  0  &  2        & 2t+1 & t+1 & 3t+5      & 3t+7   \\
(H0)   &  2  &  1  &  0  &  2        & 2t   & t   & 3t+3      & 3t+5    \\
(I0)   &  1  &  2  &  0  &  2        & 2t+1 & t   & 3t+4      & 3t+6    \\
(J0)   &  0  &  3  &  0  &  2        & 2t+2 & t   & 3t+5      & 3t+7     \\
\hline
(A1)   &  0  &  0  &  1  &  7        & 2t+3 & t+4 & 3t+1      & 3t+9    \\
(B1)   &  1  &  0  &  1  &  5        & 2t+1 & t+2 & 3t        & 3t+6    \\
(C1)   &  0  &  1  &  1  &  5        & 2t+2 & t+2 & 3t+1      & 3t+7     \\
(D1)   &  2  &  0  &  1  &  3        & 2t+1 & t+1 & 3t+2      & 3t+6    \\
(E1)   &  1  &  1  &  1  &  3        & 2t   & t   & 3t        & 3t+4    \\
(F1)   &  0  &  2  &  1  &  3        & 2t+1 & t   & 3t+1      & 3t+5     \\
(G1)   &  3  &  0  &  1  &  1        & 2t+1 & t   & 3t+4      & 3t+6       \\
(H1)   &  2  &  1  &  1  &  1        & 2t+2 & t   & 3t+5      & 3t+7     \\
(I1)   &  1  &  2  &  1  &  1        & 2t+3 & t   & 3t+6      & 3t+8      \\
(J1)   &  0  &  3  &  1  &  1        & 2t+4 & t   & 3t+7      & 3t+9     \\
\hline
(A2)   &  0  &  0  &  2  &  6        & 2t+3 & t+3 & 3t        & 3t+8    \\
(B2)   &  1  &  0  &  2  &  4        & 2t+3 & t+2 & 3t+2      & 3t+8   \\
(C2)   &  0  &  1  &  2  &  4        & 2t+2 & t+1 & 3t        & 3t+6    \\
(D2)   &  2  &  0  &  2  &  2        & 2t+1 & t   & 3t+1      & 3t+5    \\
(E2)   &  1  &  1  &  2  &  2        & 2t+2 & t   & 3t+2      & 3t+6    \\
(F2)   &  0  &  2  &  2  &  2        & 2t+3 & t   & 3t+3      & 3t+7    \\
\hline
\end{tabular}
\end{center}
\caption{Possible parameters}
\label{table:para}
\end{table}

Recall that the diameter $D$ of a graph $H$ is at most 
the number of distinct eigenvalues 
of $Q(H)$ minus one. 
Therefore $D$ is at most $5$. 
By Lemma \ref{substr} (3), 
we consider the following three cases: 


\begin{description}
\item[{\bf Case 1:}] $a_{(1,0)} = 0$. 
\end{description}

By Proposition \ref{propjack} (3), 
we have either $H=H_2$ or $H=H_4$ or 
there exists an edge $xy$ such that 
both $x$ and $y$ have degree two 
and $n:=|V(H)| > \frac{5}{2} a_{(3,0)}$. 
So we may assume that there exists an edge $xy$ 
such that both $x$ and $y$ have degree two. 
As $H$ is bipartite with diameter $D$ at most five, 
any vertex of $H$ lies at distance at most $D-1$ to the edge $xy$. 
For $D=3$ we obtain $n \leq 2 + 2 + 4 = 8$, 
for $D=4$ we obtain $n \leq 2 + 2 + 4 + 8 = 16$, 
and for $D=5$ we obtain $n \leq 32$ in this way. 
If $a_{(3,0)} \geq 6$, 
then $n > \frac{5}{2} a_{(3,0)} \geq 15$. 
But in the case (A0) we have $D \leq 3$, 
and in the cases (B0) and (C0) we have $D \leq 4$. 
So if $a_{(3,0)} \geq 6$ then $a_{(3,0)} = 6$ and $n=16$ 
and we have case (B0) with $t = 3$. 
But in order to obtain $n=16$ in case of (B0) 
we need four edges in side $A_{(3,0)}$. This in turn implies 
(by Proposition \ref{propjack}) 
that $n  \geq 21$, a contradiction. 
So $a_{(3,0)} \leq 4$. \\
If $a_{(3,0)} = 4$, then (as $n > \frac{5}{2} a_{(3, 0)}$) $n > 10$, 
so $n \geq 12$. In the cases (D0) and (F0)
we have $D \leq 4$, so hence $n \leq 2 + 2 + 4 + 6 = 14$ (as $a_{(3,0)} = 4$). 
This implies that case (D0) is not possible 
and in case (F0) we have $t=2$ and $n=12$. 
If there is a path of length three in the subgraph induced by $A_{(2,0)}$ 
then $n \leq 2 + 2 + 2 + 4 =10$, impossible. 
Now we contract all the vertices of $H$ to obtain $H''$ 
and for each edge $e$ of $H''$ 
we denote the number of vertices of degree $2$  contracted on $e$. 
Note that there are four possibilities for $H''$, 
but two of them are rules out by Lemma \ref{substr} (5). 
If $H''$ is $K_4$, then there is at most one edge with weight 
at least two and all weights are at most three. 
So there is only one possibility for $H$ in this case. 
One can easily check that $\cQ(H)$ has not only integral eigenvalues. 
If $H''$ has two cycles of length two, 
then those four edges must have odd weight, 
and one of them must be  of weight three. 
But then one of the other two edges have weight 2, 
and this is impossible by Lemma \ref{substr} (5).\\ 
For case (E0) we have $12 \leq n \leq 2 + 2 +4 +6 +6 = 20$, 
so this case is not possible. 
In cases (G0) and (J0) 
we have $6 \leq n \leq 2+2+4 + 4= 12$, 
and in cases (H0) and (I0) 
we obtain $6 \leq n \leq 2 + 2 + 4 + 4 + 4 = 16$. 
This means that in cases (G0), (H0), and (J0) we have $t=1$, 
and in case (I0) $t=0,1,2,3$, 
and it is easily checked that the only graph occurring is $H_3$.

\begin{description}
\item[{\bf Case 2:}] $a_{(1,0)} = 1$. 
\end{description}

Since $m_3=0$ and $|V(H)| \equiv 0 \pmod 3$, 
the cases (A1), (B1), (D1), and (G1) do not happen 
by Corollary \ref{cor:235}. 
If $D \leq 4$ then $n \leq 1 + 1 + 2 + 4 + 4 =12$ as $H$ is bipartite 
and $a_{(3,0)} = 1$. 
Also $a'_{(3,0)} = a_{(3,0)} - 1$, so in case (C1) 
we obtain $12 \geq n > 10$, and hence this case is not possible. 
In case (E1) we obtain $5 +1 < n \leq 1 +1 + 2 + 4 + 4 + 2 = 14$, 
so $t = 2$ and $n = 10$. 
In case (F1) we obtain $5 + 1 < n \leq 1 + 1 + 2 +4 +2 = 10$, 
so $t=1$ and $n=8$. 
And in both cases it is easy to check that they do not occur.
For cases (H1)-(J1), $n \leq 9$, so $n = 8$ or $n=9$. 
In both cases, it is easy to check there is no graph $H$.

\begin{description}
\item[{\bf Case 3:}] $a_{(1,0)} = 2$. 
\end{description}

It follows from Lemma \ref{substr} that 
the two vertices in $A_{(1,0)}$ are at distance $2$. 
It is easy to check that the diameter three can not occur, and $n \geq 7$. 
This rules out case (A2). 
For diameter $4$, we obtain $n \leq 2 + 1 +1 + 2 +2 = 8$ 
and for $D =5$, $n \leq 10$. 
This means that 
for case (B2) $t=0$ and $n =8$, 
case (C2) can not occur, 
for case (D2) we have $t=1$ and $n=8$, 
for case (E2) $t=1$ and $n=9$, 
and case (F2) is not possible. 
Case (B2) is not possible as if we look at the subgraph $H'$ 
by removing the vertices of degree $2$ 
we see that this subgraph has to have two vertices of degree 3 
and hence at least $6$ vertices. But this means that $n \geq 9$, 
a contradcition. 
It is easy to check that the two remaining cases are not possible. \\

This completes the proof of Proposition \ref{prop:gr0ev}.  
\end{proof}

\subsubsection{The case where $0 \not\in \Ev(\cQ(H,f))$} 

In this subsection we show the following proposition. 

\begin{prop}\label{prop:not0ev}
Let $(H,f)$ be a connected vertex-weighted graph with  
$5 \in \Ev(\cQ(H,f)) \subseteq \{0,1,2,3,4,5\}$. 
If $0 \not\in \Ev(\cQ(H,f))$, 
then $(H,f)$ is one of the five vertex-weighted graphs 
in Figure \ref{fig:H6-9}. 
\end{prop}

\begin{figure}[h]
\begin{center}
\begin{tabular}{ccccc}
\includegraphics[scale=0.5]
{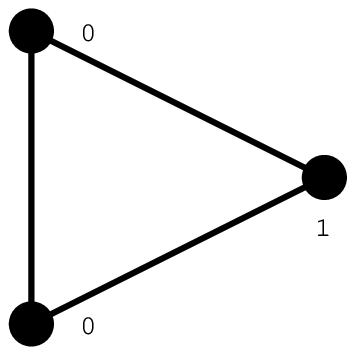} &
\includegraphics[scale=0.5]
{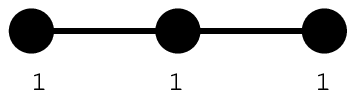} &
\includegraphics[scale=0.5]
{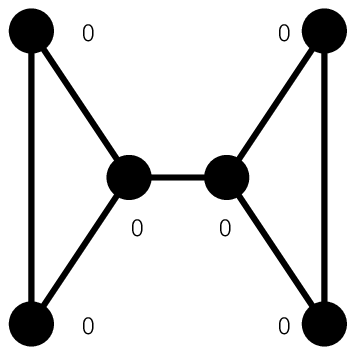} &
\includegraphics[scale=0.5]
{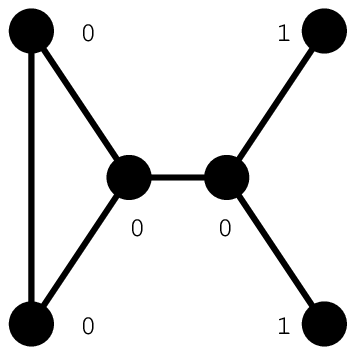} &
\includegraphics[scale=0.5]
{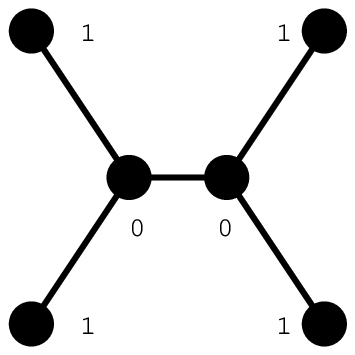} \\
$(H_5, f_5)$ &
$(H_6, f_6)$ &
$(H_7, f_7)$ &
$(H_8, f_8)$ &
$(H_9, f_9)$ \\
\end{tabular}
\end{center}
\caption{The vertex-weighted graphs $(H_5,f_5), \ldots, (H_9,f_9)$}
\label{fig:H6-9}
\end{figure}

Note that the diameter $D$ of $H$ is at most $4$ 
since $\cQ(H,f) $ has at most $5$ distinct eigenvalues.

\begin{lem}\label{lem:a10}
Let $(H,f)$ be a connected vertex-weighted graph 
with $\Ev(\cQ(H,f)) \subseteq \{\theta \in \bR \mid 1 \leq \theta \}$. 
Then $A_{(1,0)}= \emptyset$.  
\end{lem}

\begin{proof}
Suppose that $A_{(1,0)} \neq \emptyset$. 
Take $x \in A_{(1,0)}$.  
Let $y$ be the vertex adjacent to $x$. 
Then 
the smallest eigenvalue of $\cQ(H,f)|_{\{x,y\}}$ is less than $1$, 
which is a contradiction to 
$\Ev(\cQ(H,f)) \subseteq \{\theta \in \bR \mid 1 \leq \theta \}$. 
Hence $A_{(1,0)}= \emptyset$. 
\end{proof}

\begin{lem}\label{lem:A21}
Let $(H,f)$ be a connected vertex-weighted graph 
with $\Ev(\cQ(H,f)) \subseteq \{ \theta \in \bR \mid \theta \leq 5 \}$. 
Then $A_{(2,1)}$ is an independent set of $H$. 
\end{lem}

\begin{proof}
Suppose that there exist two adjacent vertices 
$x$ and $y$ in $A_{(2,1)}$. 
Then the largest eigenvalue of $\cQ(H,f)|_{\{x,y\}}$ is equal to $5$ 
but $V(H) \neq \{x,y\}$, 
which is a contradiction to 
$\Ev(\cQ(H,f)) \subseteq \{\theta \in \bR \mid \theta \leq 5 \}$. 
Hence $A_{(2,1)}$ is an independent set of $H$. 
\end{proof}

\begin{lem}\label{lem:A21-2}
Let $(H,f)$ be a connected vertex-weighted graph 
with $\Ev(\cQ(H,f)) \subseteq \{\theta \in \bR \mid \theta \leq 5 \}$. 
Then there is no triangle $K_3$ consisting three vertices of types 
$(2,0),(2,1),(3,0)$. 
\end{lem}

\begin{proof}
If there is a triangle $K_3$ consisting three vertices of types 
$(2,0),(2,1),(3,0)$, 
then  
$\cQ(H,f)|_{K_3}$ has the largest eigenvalue greater than $5$, 
which is a contradiction. 
\end{proof}

\begin{lem}\label{lem:Dleq3}
Let $(H,f)$ be a connected vertex-weighted graph 
with $5 \in \Ev(\cQ(H,f)) \subseteq \{1,2,3,4,5\}$. 
If $a_{(1,1)}=0$ and $a_{(3,0)} \geq 1$, then 
the diameter of $H$ is at most $3$. 
\end{lem}

\begin{proof}
Since $m_0=0$, $a_{(1,0)}=0$, and $a_{(1,1)}=0$, 
we have $-2 m_2 - 2 m_3 +4 = a_{(3,0)}+ a_{(2,1)}$ 
by Corollary \ref{cor:tr-8}. 
Since $m_2, m_3, a_{(3,0)}, a_{(2,1)}$ are nonnegative integers, 
if $a_{(3,0)} \geq 1$, then $m_2=0$ or $m_3=0$. 
Therefore $D \leq 3$. 
\end{proof}


\begin{prop}\label{prop:H368}
Let $(H,f)$ be a connected vertex-weighted graph 
with $5 \in \Ev(\cQ(H,f)) \subseteq \{1,2,3,4,5\}$. 
If $a_{(1,1)}=0$, 
then $(H,f)$ is $(H_5,f_5)$ or $(H_7,f_7)$. 
\end{prop}

\begin{proof}
If $a_{(3,0)} = 0$, then $H$ is an $n$-gon and 
by Lemma \ref{substr} (4), and Lemma \ref{lem:A21}, 
either $n=2a_{(2,0)}$ or $n =3$ and $(H,f) = (H_5, f_5)$. 
By Corollary \ref{cor:tr-8}, in the first case we have 
$a_{(2,1)} = a_{(2,0)} = 2$ and $H$ is a quadrangle. 
It is easy to check that this is not possible. 
If $a_{(3,0)} \geq 1$, then as $a_{(3,0)}$ is even $a_{3,0)} \geq 2$. 
As then $H$ has at least two cycles,
two degree 3 vertices must be adjacent by Lemma \ref{substr} (5). 
As the diameter is at most three it is now 
easy to check that we must have $(H,f) = (H_7, f_7)$. 
This completes the proof. 
\end{proof}

\begin{prop}\label{prop:change}
Let $(H,f)$ be a connected vertex-weighted graph. 
Suppose that $(H,f)$ has a triangle $x_1x_2x_3$ 
such that $x_1$ and $x_2$ are vertices of type $(2,0)$. 
Let $(\tilde{H},\tilde{f})$ be the vertex weighted graph 
obtained from $(H,f)$ by deleting the edge $x_1x_2$ and 
changing the type of the vertices $x_1$ and $x_2$ to 
type $(1,1)$. 
Then, $\Ev(\cQ(H,f)) \subseteq \bZ$ if and only if 
$\Ev(\cQ(\tilde{H},\tilde{f}))\subseteq \bZ$. 
\end{prop}

\begin{proof}
Let $M_1:=\cQ(H,f)$ and $M_2:=\cQ(\tilde{H},\tilde{f})$. 
First, suppose that $\Ev(M_2) \subseteq \bZ$. 
Take any $\theta \in \Ev(M_1)$. 
Then  
$M_1u =\theta u$ for some $0 \neq u \in \bR^n$.  
Therefore, we have 
$2u_1 + u_2 + u_3 = \theta u_1$ and 
$u_1 + 2u_2 + u_3 = \theta u_2$. 
So we have 
$u_1 - u_2 = \theta(u_1 - u_2)$. 
Thus if $\theta \neq 0,1$, then $u_1=u_2$. 
In this case, it holds that $M_2u =\theta u$, 
i.e., $\theta \in \Ev(M_2) \subseteq \bZ$
Hence $\Ev(M_1) \subseteq \bZ$. 
Second, suppose that $\Ev(M_1) \subseteq \bZ$. 
Take any $\tilde{\theta} \in \Ev(M_2)$. 
Then  
$M_1\tilde{u} =\tilde{\theta} \tilde{u}$ 
for some $0 \neq \tilde{u} \in \bR^n$.  
Therefore, we have 
$3\tilde{u}_1 + \tilde{u}_3 = \tilde{\theta} \tilde{u}_1$ and 
$3\tilde{u}_2 + \tilde{u}_3 = \tilde{\theta} \tilde{u}_2$. 
So we have 
$3(\tilde{u}_1 - \tilde{u}_2) = \tilde{\theta}(\tilde{u}_1 - \tilde{u}_2)$. 
Thus if $\tilde{\theta} \neq 0,3$, then $\tilde{u}_1=\tilde{u}_2$. 
In this case, it holds that $M_1\tilde{u} =\tilde{\theta} \tilde{u}$, 
i.e., $\tilde{\theta} \in \Ev(M_1) \subseteq \bZ$. 
Hence $\Ev(M_2) \subseteq \bZ$. 
\end{proof}

\begin{figure}[h]
\begin{center}
\begin{tabular}{ll}
$M_1 = 
\left(
\begin{array}{cc|cccc}
2 & 1 & 1 & 0 & \cdots & 0 \\
1 & 2 & 1 & 0 & \cdots & 0 \\
\hline 
1 & 1 & w_3  & a_{3,4}  & \cdots & a_{3,n} \\
0 & 0 & a_{4,3}  & w_4 &      & a_{4,n}  \\
\vdots & \vdots  & \vdots   &   & \ddots & \vdots \\
0 & 0 & a_{n,3}  & a_{n,4}  &   \cdots  & w_n  \\
\end{array}
\right)$ 
 & 
$M_2 = 
\left(
\begin{array}{cc|cccc}
3 & 0 & 1 & 0 & \cdots & 0 \\
0 & 3 & 1 & 0 & \cdots & 0 \\
\hline 
1 & 1 & w_3  & a_{3,4}  & \cdots & a_{3,n} \\
0 & 0 & a_{4,3}  & w_4 &      & a_{4,n}  \\
\vdots & \vdots  & \vdots   &   & \ddots & \vdots \\
0 & 0 & a_{n,3}  & a_{n,4}  &   \cdots  & w_n  \\
\end{array}
\right)$ \\
\end{tabular}
\end{center}
\caption{The matrices in the proof of Proposition \ref{prop:change}}
\label{tab:mat-change}
\end{figure}

\begin{cor}\label{cor:H689}
Let $(H,f)$ be one of the connected vertex-weighted graphs 
$(H_6,f_6)$ or $(H_8,f_8)$ or $(H_9,f_9)$. 
Then $(H,f)$ satisfies 
$5 \in \Ev(\cQ(H,f)) \subseteq \{1,2,3,4,5\}$. 
\end{cor}

\begin{proof}
This follows from Propositions \ref{prop:H368} and \ref{prop:change}. 
\end{proof}

Now we assume that $A_{(1,1)} \neq \emptyset$ and that $d_H(x,y) \geq 3$ 
for any distinct vertices $x$ and $y$ in $A_{(1,1)}$ 
(Proposition \ref{prop:change}).

\begin{prop}\label{prop:none}
There does not exist connected vertex-weighted graph $(H,f)$ 
satisfying $5 \in \Ev(\cQ(H,f)) \subseteq \{1,2,3,4,5\}$ 
such that $d_H(x,y) \geq 3$ for any distinct vertices $x$ and $y$ 
in $A_{(1,1)}$ and $a_{(1,1)} \geq 1$.
\end{prop}

\begin{proof}
We prove the proposition by contradiction. 
Suppose that there exists 
a connected vertex-weighted graph $(H,f)$ 
satisfying $5 \in \Ev(\cQ(H,f)) \subseteq \{1,2,3,4,5\}$ 
such that 
$d_H(x,y) \geq 3$ for any distinct vertices $x$ and $y$ 
in $A_{(1,1)}$ and $a_{(1,1)}\geq 1$. 
If the diameter $D$ of $H$ equals $2$, then $n:=|V(H)| \leq 4$ 
and it is easy to check that there are no such graphs. 
So we have $D \geq 3$. 
Therefore $m_2 +m_3 \geq 1$, 
and so $a_{(3,0)} \leq a_{(1,1)} +2$ 
by Corollary \ref{cor:tr-8}. 
If $a_{(3,0)} = a_{(1,1)} + 2$, 
then $m_2 + m_3 = 1$, $D=3$, $a_{(2,1)} = 0$
and $n \leq 1 + 1 + 2 + 4 =8$. 
So we have $(a_{(3,0)}, a_{(1,1)}) \in \{(3,1), (4,2), (5,3) \}$. 
The case $(a_{(3,0)}, a_{(1,1)}) = (5,3)$ would be a contradiction 
to the assumption that $d_H(x,y) \geq 3$ 
for any distinct vertices $x$ and $y$ in $A_{(1,1)}$. 
The case $(a_{(3,0)}, a_{(1,1)}) = (3,1)$ gives $m_1 = m_2 =0$, 
which is a contradiction to $D = 3$. 
For the case $(a_{(3,0)}, a_{(1,1)}) = (4,2)$,  
there is no solution. 
So this shows $a_{(1,1)} \geq a_{(3,0)}$. 

If $a_{(1,1)} \geq a_{(3,0)} + 2$ then a neighbour of some vertex 
in $A_{(1,1)}$ has degree 2, so $n \leq 8$, but because of the assumption 
that $d_H(x,y) \geq 3$ for any distinct vertices $x$ and $y$ in $A_{(1,1)}$, 
we find $( a_{(3,0)}, a_{(1,1)} ) \in \{ (1, 3), (0, 2) \}$. 
As $n \equiv a_{(1,1)} + m_3 \pmod 3$, 
it is easy to check that there are no possibilities. 
Therefore $a_{(1,1)} = a_{(3,0)} \geq 1$. Then $a_{(2,1)} + m_2 + m_3 = 2$. 

Now we consider the case $D = 4$.
If $D = 4$, then $m_3 = m_2 =1 $ and $a_{(2,1)} = 0$. 
So $n \equiv a_{(1,1)} + 1 \pmod 3$. 
The case  $a_{(1,1)} = 1= a_{(3,0)}$ is not possible 
as then there are $2$ edges in 
the subgraph of $H$ induced by 
the set $A_{(2,0)}$ 
and only one vertex of degree three. 
Now $a_{(1,1)}= 2$ implies $n \in \{6, 9\}$, 
$a_{(1,1)}=3$ implies $n \in \{7, 10\}$, 
$a_{(1,1)}=4$ implies $n \in \{10, 13\}$, 
and $a_{(1,1)}= 5$ implies $n \in \{12, 15\}$. 
In all the cases 
it is easy to check that they do not occur. 
And clearly $a_{(1,1)} \geq 6 $ is impossible. 

So this shows that $D = 3$. 
Now $n \leq 8$. 
And in similar fashion one can show that no case can occur.
\end{proof}

\begin{proof}[Proof of Proposition \ref{prop:not0ev}]
It follows from Proposition \ref{prop:H368}, 
Corollary \ref{cor:H689}, and Proposition \ref{prop:none}. 
\end{proof}

\subsection{The case of exceptional graphs}\label{sec:proof-EG} 

In this subsection, we show the following: 

\begin{thm}\label{thm:EG001}
Let $\Gamma$ be a connected integral exceptional graph 
with spectral radius three. 
Then, $\Gamma$ is isomorphic to one of the $13$ graphs 
in Figure~\ref{fig:EG}.
\end{thm}

Now we recall some definitions and results. 
Let $|\triangle(\Gamma)|$ denote the number of 
triangles in a graph $\Gamma$. 

\begin{prop}[{cf. \cite[Corollary 8.1.3]{god}}]
Let $\Gamma$ be a graph. 
Then 
$\tr(A(\Gamma))^0 = |V(\Gamma)|$,
$\tr(A(\Gamma))^1 = 0$, 
$\tr(A(\Gamma))^2 = 2|E(\Gamma)|$, and 
$\tr(A(\Gamma))^3 = 6|\triangle(\Gamma)|$. 
\end{prop}

\begin{cor}
Let $\Gamma$ be a connected integral graph with 
smallest eigenvalue at least $-2$ and 
largest eigenvalue $3$. 
Let $m_r$ denote the multiplicity of $r \in \bR$ 
as an eigenvalue of $A(\Gamma)$. 
Then the following hold:
\begin{eqnarray}
\label{eq:A-m-0}
1 + m_{2} + m_{1} + m_{0} + m_{-1} + m_{-2} &=& |V(\Gamma)|, \\
\label{eq:A-m-1}
3 + 2 m_{2} + m_{1} - m_{-1} -2 m_{-2} &=& 0, \\
\label{eq:A-m-2}
9 + 4 m_{2} + m_{1} + m_{-1} +4 m_{-2} &=& 2|E(\Gamma)|, \\
\label{eq:A-m-3}
27 + 8 m_{2} + m_{1} - m_{-1} -8 m_{-2} &=& 6|\triangle(\Gamma)|. 
\end{eqnarray}
\end{cor}

\begin{df}\label{df:001}
Let $\Gamma$ be a graph with $V(\Gamma)=\{1,\ldots,n\}$. 
Let $P$ be the orthognal projection of ${\mathbb R}^n$ 
	onto ${\mathscr E}(\mu)$,
	where ${\mathscr E}(\mu)$ is the eigenspace of $A(\Gamma)$
	for the eigenvalue $\mu$ of $A(\Gamma)$.
Then a subset $X$ of $V(\Gamma)$ satisfying the following condition
	is called a star set for $\mu$ of $\Gamma$:
\begin{equation}\label{eq:001}
\text{the vectors $P{\bf e}_j$ ($j\in X$) form
	a basis for ${\mathscr E}(\mu)$,}
\end{equation}
where $\{{\bf e}_1,\ldots,{\bf e}_n\}$
	is the standard basis of ${\mathbb R}^n$.
\end{df}

\begin{df}\label{df:002}
Let $\Gamma$ be a graph with $V(\Gamma)=\{1,\ldots,n\}$
	and an eigenvalue $\mu$.
Let $X$ be a star set for $\mu$ of $\Gamma$.
Then the subgraph $\Gamma-X$ of $\Gamma$ is called
	the star complement for $\mu$
	corresponding to $X$.
\end{df}

Let $\Gamma$ be a graph with adjacency matrix
$\begin{pmatrix}
A_X & B^T\\
B&C
\end{pmatrix}$,
where $X$ is a star set for an eigenvalue $\mu$ of $\Gamma$.
Then we define a bilinear form on ${\mathbb R}^{n-|X|}$
	by $\subg{{\mathbf x},{\mathbf y}}{X}
	={\mathbf x}^T(\mu I-C)^{-1}{\mathbf y}$,
	and denote the columns of $B$
	by ${\mathbf b}_{v}$ ($v\in X$).

\begin{thm}[\cite{tech2}]\label{thm:002}
Suppose that $\mu$ is not an eigenvalue of the graph $\Gamma'$.
Then there exists a graph $\Gamma$
	with a star set $X$ for $\mu$
	such that $\Gamma -X=\Gamma'$
	if and only if
	the characteristic vectors ${\mathbf b}_v$ ($v\in X$)
	satisfy
\begin{enumerate}[(i)]
\item $\subg{{\mathbf b}_v,{\mathbf b}_v}{X}=\mu$ for all $v\in X$,
\item $\subg{{\mathbf b}_u,{\mathbf b}_v}{X}\in\{-1,0\}$
	for all pairs $u,v$ of distinct vertices in $X$.
\end{enumerate}
\end{thm}

If $\Gamma$ has $\Gamma'$ as a star complement for $\mu$
	with corresponding star set $X$,
	then each induced subgraph $\Gamma -Y$ ($Y\subset X$)
	also has $\Gamma'$ as a star complement for $\mu$.

By the star complement technique (see, for example, \cite{tech2}),
	we determine all connected exceptional graphs $\Gamma$
	satisfying 
$3 \in \Ev(A(\Gamma)) \subseteq \{-2,-1,0,1,2,3 \}$. 	

By $G(\Gamma')$,
	we define the graph satisfying the following conditions:
\begin{enumerate}[(i)]
\item the vertices are the $(0,1)$-vectors $\mathbf b$ in ${\mathbb R}^{t}$
	such that $\subg{{\mathbf b},{\mathbf b}}{\Gamma'}=-2$,
	where $t=|V(\Gamma')|$,
\item ${\mathbf b}_1$ is adjacent to ${\mathbf b}_2$
	if and only if
	$\subg{{\mathbf b}_1,{\mathbf b}_2}{\Gamma'}\in\{-1,0\}$.
\end{enumerate}
A graph $\Gamma$ with a star set $X$ for $-2$
	such that $\Gamma -X=\Gamma'$ now corresponds to a clique
	in $G(\Gamma')$.
There exist $573$ graphs such that they are connected exceptional
	and have the smallest eigenvalues greater than $-2$
	(see \cite{new}).
There are $20$ such graphs on $6$ vertices, 
	$110$ on $7$ vertices
	and $443$ on $8$ vertices. 

Since the connected exceptional graphs
	with smallest eigenvalue $-2$ have subgraphs
	isomorphic to one of such graphs as a star complement for $-2$,
	we can obtain the complete list of exceptional graphs 
satisfying 
$3 \in \Ev(A(\Gamma)) \subseteq \{-2,-1,0,1,2,3 \}$  
from $573$ such graphs.
By comupter, we obtain the following lemma:
	
\begin{lem}\label{lm:006}
Let $\Gamma$ be a connected exceptional graph satisfying 
$3 \in \Ev(A(\Gamma)) \subseteq \{-2,-1,0,1,2,3 \}$. 
If $|V(\Gamma)|\le 12$,
	then $\Gamma$ is isomorphic to one of the graphs
	in Figure~\ref{fig:EG}.
\end{lem}

In the following, we show that 
any connected exceptional graph satisfying 
$3 \in \Ev(A(\Gamma)) \subseteq \{-2,-1,0,1,2,3 \}$. 
has at most $12$ vertices.

\begin{lem}\label{lm:003}
Let $\Gamma$ be a connected exceptional graph satisfying 
$3 \in \Ev(A(\Gamma)) \subseteq \{-2,-1,0,1,2,3 \}$. 
If $|\triangle(\Gamma)|=0$, 
then $\Gamma$ is the Petersen graph. 
In particular, $m_{-2}=4$. 
\end{lem}

\begin{proof}
If $\Gamma$ contains an induced $K_{1,4}$, 
then, by Perron-Frobenius Theorem, 
$\Gamma$ can not contain an induced bipartite subgraph 
containing $K_{1,4}$ and therefore $\Gamma$ is $K_{1,4}$, 
but this is impossible as it is bipartite and 
hence spectral radius is not three. 
This means that $\Gamma$ has maximum degree at most three 
and hence as it has spectral radius three, it must be three-regular. 
So $2|E(\Gamma)|=3|V(\Gamma)|$. 
If $\Gamma$ contains an induced quadrangle, 
then again, by Perron-Frobenius Theorem, 
$\Gamma$ must be this quadrangle, a contradiction. 
By solving Equations (\ref{eq:A-m-0})-(\ref{eq:A-m-3}) with 
$2|E(\Gamma)|=3|V(\Gamma)|$ and $|\triangle(\Gamma)|=0$, 
we have 
$m_{-1} = m_0 = m_2 = 0$, $m_1 = 5$, $m_{-2} = 4$ and $n=10$. 
Thus it follows that $\Gamma$ is the Petersen graph. 
\end{proof}

\begin{lem}\label{lm:005}
Let $\Gamma$ be a connected exceptional graph satisfying 
$3 \in \Ev(A(\Gamma)) \subseteq \{-2,-1,0,1,2,3 \}$. 
Then $|V(\Gamma)| \leq 12$. 
\end{lem}

\begin{proof}
First, we show that $\sum_{i=-1}^3 m_i \leq 8$.
There exists a star set for $-2$ of $\Gamma$
	such that $\Gamma -X$ is exceptional,
	i.e., $\theta_{\min}>-2$.
Then $|V(\Gamma)|-|X|=6,7$ or $8$ (see \cite{new}).
Therefore $\sum_{i=-1}^3 m_i = \sum_{i=-2}^3 m_i - m_{-2} 
=|V(\Gamma)|-|X|\le 8$.
Hence $\sum_{i=-1}^3 m_i \leq 8$.

Second, we show that $m_{-2} \le 4$. 
If $|\triangle(\Gamma)|=0$, 
then $m_{-2}=4$ by Lemma \ref{lm:003}. 
So we assume that $|\triangle(\Gamma)| \geq 1$.  
First, we show that 
$4m_2 + m_{-2} + m_1 + m_{-1} \le 15$. 
It is well-known that 
$\theta_{\max}(A(\Gamma)) \geq 
\frac{1}{|V(\Gamma)|} \sum_{v\in V(\Gamma)} \deg_{\Gamma}(v)$. 
By Equation (\ref{eq:A-m-2}) 
and $2|E(\Gamma)|=\sum_{v\in V(\Gamma)}\deg_{\Gamma} (v)$, 
we have $9+\sum_{i=-2}^2 i^2 m_i \le 3|V(\Gamma)|$.
By Equation (\ref{eq:A-m-0}) and $\sum_{i=-1}^3 m_i \leq 8$, 
we have $|V(\Gamma)| \leq 8+m_{-2}$. 
Therefore, $9 + \sum_{i=-2}^2 i^2 m_i \le 3(8+m_{-2})$, that is, 
$\sum_{i=-1}^{2} i^2 m_i + 4 m_{-2} \le 15 + 3m_{-2}$.  
Hence we have 
$4m_2 + m_1 + m_{-1} + m_{-2}  \le 15$.

By calculating [Equation (\ref{eq:A-m-3})] $-$ [Equation (\ref{eq:A-m-1})], 
we obtain $6m_2 - 6m_{-2}+24=6|\triangle(\Gamma)|$, that is, 
$m_{2} = m_{-2} + |\triangle(\Gamma)| - 4$.
By calculating 
 [Equation (\ref{eq:A-m-3})] $-$ $4 \times$ [Equation (\ref{eq:A-m-1})], 
we obtain $- 6m_1 + 6m_{-1}+ 15=6|\triangle(\Gamma)|$, that is, 
$m_{-1} =m_{1} + 2|\triangle(\Gamma)| - 5$.
Thus we have  
$m_{-2} \leq \frac{1}{5}(36 - 6 |\triangle(\Gamma)| - 2m_1)$.
If $|\triangle(\Gamma)|=1$,
	then $m_1 = m_{-1} + 5-2 \ge 3$ 
	and thus $m_{-2} \le \frac{24}{5}$. 
If $|\triangle(\Gamma)| \ge 2$,
	then 
	$m_{-2} \le \frac{24}{5}$.
Hence $m_{-2} \le 4$.

By Equation (\ref{eq:A-m-0}), 
$|V(\Gamma)|=\sum_{i=-1}^3m_i + m_{-2} \le 8+4=12$. 
Hence the lemma holds.
\end{proof}

\begin{proof}[Proof of Theorem \ref{thm:EG001}]
It follows from Lemmas~\ref{lm:006} and \ref{lm:005}.
\end{proof}

\section{Concluding Remarks}

In this paper, we classified the connected non-bipartite integral graphs 
with spectral radius three. 
They have at most $13$ vertices. 
A natural question is given the set of eigenvalues of a connected graph 
what one can say about the number of vertices, 
the degree sequence etcetera. 
A bound on the number of vertices given the diameter and spectral radius 
is given in \cite{CDKL}. 
Although it is believed that this bound is asymptotically good, 
for small spectral radius, it is not a good bound. 

\begin{description}
\item
{\bf Challenge 1.}  
Classify the connected integral bipartite graphs.
\end{description}

\noindent
Brouwer and Haemers \cite{BHtree} classified the integral trees 
with spectral radius three, 
and K.~Bali{\'n}ska {\it et al.} did some work 
on the bipartite non-regular integral graphs with maximum degree four 
\cite{BS}, \cite{BSZ}. 
It seems that the general case is not doable without 
a better bound on the number of vertices. 
Probably the methods in this paper can be extended 
to find all integral graph with spectral radius four 
and smallest eigenvalue $-2$. 

\begin{description}
\item
{\bf Challenge 2.} 
Classify the integral graphs with spectral radius four 
and smallest eigenvalue $-2$. 
\end{description}

\noindent
Some work towards this challenge has been done by \cite{SS}.

\section*{Acknowledgements}

This work is partially based on the master thesis of the first author. 
This work was made possible partially by 
a NRF-JSPS grant (F01-2009-000-10113-0). 
The second author was supported 
by the Basic Science Research Program through 
the National Research Foundation of Korea (NRF) 
funded by the Ministry of Education, Science and Technology 
(2010-0008138). 
The third author was supported 
by Priority Research Centers Program through 
the National Research Foundation of Korea (NRF) 
funded by the Ministry of Education, Science and Technology 
(2010-0029638). 
All the supports were greatly appreciated.



\begin{thebibliography}{99}

\bibitem{BCRSS} 
{K.~Bali{\'n}ska, D.~Cvetkovi{\'c}, Z.~Radosavljevi{\'c}, 
S.~Simi\'{c}, and D.~Stevanovi\'{c}}:
{A survey on integral graphs}, 
{\it Univerzitet u Beogradu. Publikacije Elektrotehni\v ckog
 Fakulteta. Serija Matematika} 
{\bf 13} (2002) 42--65.

\bibitem{BS} 
{K.~Bali{\'n}ska, S.~Simi{\'c}}: 
{The nonregular, bipartite, integral graphs with maximum degree 
4. {I}. {B}asic properties}, 
{\it Discrete Mathematics} 
{\bf 236} (2001) 13--24.

\bibitem{BSZ} 
{K.~Bali{\'n}ska, S.~Simi{\'c}, K.~Zwierzy{\'n}ski}: 
{Which non-regular bipartite integral graphs with maximum
degree four do not have {$\pm1$} as eigenvalues?}, 
{\it Discrete Mathematics} 
{\bf 286} (2004) 15--24. 

\bibitem{BHtree} 
{A. E. Brouwer and W. H. Haemers}: 
{The integral trees with spectral radius 3}. 
{\it Linear Algebra and its Applications} 
{\bf 429} (2008) 2710--2718. 

\bibitem{BC} 
{F.~C.~Bussemaker and D.~M.~ Cvetkovi\'{c}}: 
{There are exactly 13 connected, cubic, integral graphs}, 
{\it Univerzitet u Beogradu. Publikacije Elektrotehni\v ckog
 Fakulteta. Serija Matematika i Fizika} 
{\bf 544-576} (1976) 43--48. 


\bibitem{root} 
{P.~J.~Cameron, J.~M.~Goethals, J.~J.~Seidel, and E.~E.~Shult}: 
{Line graphs, root systems, and elliptic geometry}, 
{\it Journal of Algebra} 
{\bf 43} (1976) 305--327.

\bibitem{CDKL} 
{S.~Cioab\u{a}, E.~van Dam, J. H. Koolen, J.-H. Lee}: 
{A lower bound for the spectral radius of graphs with fixed diameter}, 
{\it European Journal of Combinatorics} 
{\bf 31} (2010) 1560--1566.

\bibitem{cve} 
{D.~M.~Cvetkovi\'{c}}: 
{Cubic integral graphs}, 
{\it Univerzitet u Beogradu. Publikacije Elektrotehni\v ckog
 Fakulteta. Serija Matematika i Fizika} 
{\bf 498-541} (1975) 107--113. 

\bibitem{tech2} 
{D.~Cvetkovi\'{c}, M.~Lepovi\'{c}, P.~Rowlinson, and S.~K.~Simi\'{c}}: 
{The maximal exceptional graphs}, 
{\it Journal of Combinatorial Theory. Series B} 
{\bf 86} (2002) 347--363.

\bibitem{new} 
{D.~Cvetkovi\'{c}, P.~Rowlinson, and S.~K.~Simi\'{c}}:
{\it Spectral Generalizations of Line Graphs 
--- On Graphs with Least Eigenvalue -2}, 
{London Mathematical Society Lecture Note Series} 
{\bf 314}, 
(Cambridge University Press, 2004). 

\bibitem{CRS2} 
{D.~Cvetkovi\'{c}, P.~Rowlinson, and S.~K.~Simi\'{c}}: 
{Signless {L}aplacians of finite graphs},
{\it Linear Algebra and its Applications} 
{\bf 423} (2007) 155--171.

\bibitem{towards1} 
{D. Cvetkovi{\'c} and S. K. Simi{\'c}}: 
{Towards a spectral theory of graphs based on the signless {L}aplacian. {I}}, 
{\it Institut Math\'ematique. Publications. Nouvelle S\'erie} 
{\bf 85(99)} (2009) 19--33. 


\bibitem{towards2} 
{D. Cvetkovi{\'c} and S. K. Simi{\'c}}: 
{Towards a spectral theory of graphs based on the signless {L}aplacian. {II}}, 
{\it Linear Algebra and its Applications} 
{\bf 432} (2010) 2257--2272. 


\bibitem{towards3} 
{D. Cvetkovi{\'c} and S. K. Simi{\'c}}: 
{Towards a spectral theory of graphs based on the 
signless {L}aplacian. {III}},  
{\it Applicable Analysis and Discrete Mathematics} 
{\bf 4} (2010) 156--166. 

\bibitem{god} 
{C. Godsil and G. Royle}: 
{\it Algebraic Graph Theory}, 
(GTM 207, Springer, New York, 2001). 

\bibitem{HS} 
{F.~Harary and A.~J.~Schwenk}: 
{Which graphs have integral spectra?}, 
{\it Graphs and combinatorics 
(Proc. Capital Conf., George Washington Univ., Washington, D.C., 1973)}, 
{\it Lecture Notes in Math.} {\bf 406}, 
Springer, Berlin, (1974) 45--51. 

\bibitem{Hoff} 
{A.~J.~Hoffman}: 
{$-1-\sqrt{2}$}, 
{\it Combinatorial structures and their applications 
(Proc. Calgary International Conf. Combin. structures and their appl., 
June 1969)} (Eds. R. Guy et al.), Gordon and Breach, New York, (1970) 173--176. 

\bibitem{RS} 
{Z. Radosavljevi{\'c} and S. Simi{\'c}}: 
{There are just thirteen connected nonregular nonbipartite 
integral graphs having maximum vertex degree four (a shortened report)}, 
{\it Graph theory ({D}ubrovnik, 1985)},
Univ. Novi Sad, Novi Sad,
(1986) 183--187. 

\bibitem{Schw} 
{A.~J.~Schwenk}: 
{Exactly thirteen connected cubic graphs have integral spectra}, 
{\it Theory and applications of graphs 
(Proc. Internat. Conf., Western Mich. Univ., Kalamazoo, Mich., 1976)}, 
{\it Lecture Notes in Math.} {\bf 642}, 
Springer, Berlin, (1978) 516--533.

\bibitem{SR} 
{S.~Simi\'{c} and Z.~Radosavljevi\'{c}}: 
{The nonregular, nonbipartite, integral graphs with maximum degree four}, 
{\it Journal of Combinatorics, Information \& System Sciences} 
{\bf 20} (1995) 9--26.
 
\bibitem{SS} 
{S. K. Simi{\'c} and Z. Stani{\'c}}: 
{{$Q$}-integral graphs with edge-degrees at most five}, 
{\it Discrete Mathematics} 
{\bf 308} (2008) 4625--4634. 

\bibitem{MAGMA}  
The Magma Computational Algebra System for Algebra, 
Number Theory and Geometry. 
URL: http://magma.maths.usyd.edu.au/magma/


\end{thebibliography}
\end{document}